\documentclass[11pt,a4paper]{article}
\usepackage[utf8]{inputenc}
\usepackage[T1]{fontenc}
\usepackage[marginparwidth=2.5cm]{geometry}
\usepackage[dvipsnames]{xcolor}
\usepackage{libertine}
\usepackage{amsmath,amssymb,amsthm,amsfonts,stmaryrd}
\usepackage{dsfont,mathrsfs}
\usepackage{mathtools}
\usepackage[colorlinks=true,citecolor=blue,linkcolor=blue]{hyperref}
\usepackage{enumitem}
\usepackage{bm}
\usepackage[todonotes={textsize=scriptsize}]{changes}

\newtheorem{dfn}{Definition}
\newtheorem{thm}{Theorem}
\newtheorem{rmk}{Remark}
\newtheorem{lem}{Lemma}

\newtheorem{prop}{Proposition}

\newtheorem{cor}{Corollary}

\newcommand{\ffi}{\varphi}
\newcommand{\dd}{\mathrm{d}}
\newcommand{\C}{\mathrm{C}}
\newcommand{\1}{\bm{1}}
\newcommand{\PP}{\mathbb{P}}
\newcommand{\EE}{\mathbb{E}}
\newcommand{\NN}{\mathbb{N}}
\newcommand{\RR}{\mathbb{R}}

\newcommand{\T}{\mathbb{T}}
\newcommand{\ZZ}{\mathbb{Z}}

\newcommand{\vertiii}[1]{{\vert\kern-0.08em\vert\kern-0.08em\vert #1 \vert\kern-0.08em\vert\kern-0.08em\vert}}
\newcommand{\mean}[1]{[ #1 ]_{\T}}
\newcommand{\meanM}[1]{[ #1 ]_M}
\newcommand{\meand}[1]{[ #1 ]_{\T^d}}

\newcommand{\authora}[1]{\gdef\authora{#1}}
\newcommand{\addressa}[1]{\gdef\addressa{#1}}
\newcommand{\emaila}[1]{\gdef\emaila{#1}}

\newcommand{\authorb}[1]{\gdef\authorb{#1}}
\newcommand{\addressb}[1]{\gdef\addressb{#1}}
\newcommand{\emailb}[1]{\gdef\emailb{#1}}

\newcommand{\authorc}[1]{\gdef\authorc{#1}}
\newcommand{\addressc}[1]{\gdef\addressc{#1}}
\newcommand{\emailc}[1]{\gdef\emailc{#1}}

\makeatletter
\newcommand{\@endstuff}{
  \bigskip\vspace{1ex}
  \small\scshape\noindent
  \authora\\
  \addressa\\
  E-mail address: \href{mailto:\emaila}{\texttt{\emaila}}\\
  \medskip\\
  \authorb\\
  \addressb\\
  E-mail address: \href{mailto:\emailb}{\texttt{\emailb}}\\
  \medskip\\
  \authorc\\
  \addressc\\
  E-mail address: \href{mailto:\emailc}{\texttt{\emailc}}}
\AtEndDocument{\@endstuff}
\makeatother

\authora{Vincent Bansaye}
\addressa{Centre de Mathématiques Appliquées (CMAP), École Polytechnique, CNRS, Palaiseau, France.}
\emaila{vincent.bansaye@polytechnique.edu}

\authorb{Ayman Moussa}
\addressb{Laboratoire Jacques-Louis Lions (LJLL), Sorbonne Université,  \& Département de Mathématiques et Applications (DMA), École Normale Supérieure, Université PSL, CNRS, Paris, France.}
\emailb{ayman.moussa@sorbonne-universite.fr}
 
\authorc{Felipe Muñoz-Hernández}
\addressc{Departamento de Ingenier\'ia Matem\'atica (DIM), Universidad de Chile, Santiago, Chile.\\
  Centre de Mathématiques Appliquées (CMAP), École Polytechnique, CNRS, Palaiseau, France.}
\emailc{fmunozh@dim.uchile.cl}

\title{Stability of a cross-diffusion system and approximation by repulsive random walks: a duality approach}
\author{\authora \and \authorb \and \authorc}
\date{\today}

\begin{document}

\maketitle

\begin{abstract}

    We consider  conservative cross-diffusion systems for two species where individual motion rates  depend linearly on the local density of the other species.
We develop duality estimates and  obtain stability and approximation results. 
We  first control  the time evolution of the gap between two bounded solutions  by means of its initial value. As a by product, we obtain  a uniqueness result for bounded solutions valid for any space dimension, under a non-perturbative smallness assumption. Using a discrete counterpart of our duality estimates, we prove the convergence
of   random walks with local repulsion in one dimensional discrete space  to cross-diffusion systems.
More precisely, we prove quantitative estimates for the gap between the stochastic process and the cross-diffusion system. We give first rough but general estimates; then we use the duality approach
to obtain fine estimates under less general conditions. 

\end{abstract}

\section{Introduction and notation}

Approximations of interacting large populations is motivated by physics, chemistry, biology and ecology. A famous macroscopic model was introduced by Shigesada, Kawasaki and Teramoto in \cite{SKT1979} to describe competing species which diffuse \emph{with} local repulsion. In the case of two species, it writes
\begin{equation*}
  \left\{
  \begin{aligned}
    \partial_tu - \Delta\bigl( d_1u + a_{11} u^2+a_{12}uv \bigr) &= u(r_1-s_{11}u-s_{12}v), \\
    \partial_tv - \Delta\bigl( d_2v + a_{22}v^2+a_{21}uv  \bigr) &= v(r_2-s_{21}u-s_{22}v),
  \end{aligned}
  \right.                                
\end{equation*}
where $u$ and $v$ are the densities of the two species and $d_i,r_i,a_{ij}$ and $s_{ij}$ are non-negative real numbers. Completed by initial and boundary conditions, this system (that we simply refer to as the \textit{SKT system}) offers a model for the spreading of two interacting species which mutually influence their propensity to diffuse, through the cross-diffusion terms $a_{ij}$.  The other coefficients represent either natural diffusion ($d_i$ coefficients), reproduction ($r_i$ coefficients) or competition ($s_{ij}$ coefficients).  The main motivation of \cite{SKT1979} was to propose a population dynamics model able to detect segregation, that is the existence of non-constant steady states $\overline{u}$ and $\overline{v}$ having disjoint superlevel sets of low threshold value. As a consequence of this motivation, the first mathematical results dealing with this system  focused on sufficient conditions for the coefficients to ensure existence of non-constant steady states, with a careful study of the stability of the latter. This study of possible segregation states is still active and we refer to the introduction of \cite{bks} for a nice state of the art. It is a striking fact that during its first years of existence within the mathematical community, the SKT system has not been studied through the prism of its Cauchy problem. As a matter of fact, existence of solutions has been tackled only a few years later: the first paper dealing with this issue is \cite{kim1984smooth} and explores the system under very restrictive conditions. Several attempt followed, but only with partial results. A substantial progress was achieved by Amann \cite{amann90,amann90b}, who proposed a rather abstract approach to study generic quasilinear parabolic systems. The scope of this technology goes far beyond the sole case of cross-diffusion systems. In the specific case of the SKT system, it offers existence of local (regular) solutions, together with a criteria of explosion to decide if the existence is global or not. This fundamental result of Amann has been then used by several authors to establish existence of global solutions for particular forms of the SKT system. This is done, in general, under a strong constraint on the coefficients. For instance, \cite{louwin} treats the case of equal diffusion rates in low dimension and \cite{HoanNguPha}, settles the one of triangular systems (that is, for two species, when $a_{12}a_{21}=0$). However, the general question of existence of global solution for the complete system remains open, even in low dimension. 

Another way to produce a global solution is to sacrifice the regularity of the solutions, and deal with only weak ones. This strategy relies on the so-called \emph{entropic structure} of the system: SKT systems as the one previously introduced, admit Lyapunov functionals which decay along time and whose dissipation allows to control the gradient of the solution. This method has been used successfully in \cite{chenjun} to prove, for the first time, existence of global weak solutions for the SKT system, without restrictive assumptions on its coefficients. After it first discovery in \cite{Galiano_Num_Math}, this entropic structure has been explored and generalized to several systems, allowing for the construction of global weak solutions for variants of the original SKT system (see \cite{jungel_bound_by,dlmt} and the references therein). With this low level of regularity for the solutions, uniqueness becomes an issue in itself. It has been studied either under simplifying assumptions on the system like in \cite{pham_temam_unique,chen_jun_unique} or in the weak-strong setting thanks to the use of a relative entropy (see \cite{chen_jun_unique_bis}).

\subsection{Objectives and state of the art}
This work is initially motivated by
 yet another mathematical challenge offered by the SKT system: its rigorous derivation. The diffusion operator used in the SKT system is  specific. We focus in this paper on the main difficulty  raised by  this operator, which is the non-linearity of  diffusion term. The initial   goal of the work is to approximate the conservative SKT system, without self-diffusion, that is the following one
\begin{equation}
  \label{eq:SKT:cons}          
  \left\{                                      
  \begin{aligned}
    \partial_tu - \Delta( d_1u + a_{12}uv ) &= 0, \\
    \partial_tv - \Delta( d_2v + a_{21}uv ) &= 0,
  \end{aligned}
  \right.
\end{equation}
where all the coefficients $d_i$ and $a_{ij}$ are assumed positive.
Whereas (possibly heterogeneous) diffusion of lifeless matter (\emph{e.g.} ink or any type of chemical substance) uses the Fick diffusion operator $-\textnormal{div}(\mu\nabla \cdot)$ to express the spread, SKT systems rely on the (more singular) operator $-\Delta(\mu\,\cdot)$. As it was already explained in \cite{SKT1979}, this choice of diffusion operator is at the core of the repulsive mechanism allowing the segregation to appear. However, the justification proposed in \cite{SKT1979} was rather formal, leaving open the question of the rigorous justification of  SKT systems. As far as our knowledge goes, there exist mainly three approaches for the derivation of  SKT systems  
\begin{enumerate}[label=(\roman*)]
  \item The first path was proposed in \cite{iidadiff}, where an SKT model is obtained as an asymptotic limit of a family of reaction-diffusion systems. In this approach the idea is that one of the two species exists in two states (stressed or not), and switch from one to the other with a reaction rate which diverges. This was used in \cite{iidadiff} to obtain formally a triangular cross diffusion system. This strategy has been followed with a rigorous analysis, mainly to produce triangular systems (see \cite{ariane} and references therein) and more recently for a family of "full" systems in \cite{ddj} which, however, do not include the SKT one.
  \item Another strategy was proposed by Fontbona and Méléard in \cite{FM}. The idea is to start from a stochastic population model in continuous space where the individuals' displacements depend on the presence of concurrents. Then, the large population limit (under adequate scaling) leads to a non-local cross-diffusion model. In comparison with the system \eqref{eq:SKT:cons}, the limit model rigorously derived in \cite{FM} is a lot less singular, because of several convolution kernels. It was explicitly asked in \cite{FM}, whether letting the convolution kernels vanish to the Dirac mass was handable limit or not. A first partial answer was given in \cite{moussa_nonloc_tri}, but applied for only specific triangular systems. More recently, it was discovered in \cite{dietert_moussa_nonloc} that even for the non-local systems, it is possible to ensure the persistence of the entropy
  structure, allowing to answer fully to the question of Fontbona and Méléard, at least for the standard SKT system. \\
  A little bit before \cite{dietert_moussa_nonloc} appeared, Chen \emph{et. al.} proposed another strategy in \cite{chendausjunbis} (see also \cite{chendausjun} which deals with a slightly different family of systems). It also starts from a stochastic model and makes use of an intermediate non-local one. The main difference with \cite{FM,moussa_nonloc_tri,dietert_moussa_nonloc} is that in \cite{chendausjunbis} the two asymptotics are done simultaneously (size of population to infinity and parameter of regularization to $0$). 
  This direct approach amounts to "commute" the asymptotic diagram from the stochastic model to the final PDE; this is a common feature with the current work that we will comment later on.
  
 \item The third path was proposed in \cite{DDD2019} and justifies the SKT model through a semi-discrete one. The latter is itself derived from a stochastic population model in discrete space where individuals are assumed to move by pair, in order to ensure reversibility of the process and the existence of an entropy for the limit model. In \cite{DDD2019} the link to the stochastic was done formally whereas the asymptotic analysis linking the semi-discrete model to the SKT system was proved rigorously, relying on a compactness argument which is allowed thanks to the existence of the Lyapunov functional for the semi-discrete system.
\end{enumerate}

In this paper, we are interested in   connections between  microscopic random individual-based models (or particle system) and such macroscopic deterministic dynamics, in the spirit of strategies (ii) and (iii) described above. We do not use any non-local approximated system as in \cite{FM,chendausjunbis}, being inspired instead by the semi-discrete approach proposed in \cite{DDD2019}. We consider also a discrete space and  that each species moves randomly and is only sensitive to the local size of the other species. Let us comment the main differences and novelties of this work compared to \cite{DDD2019}. First, we prove rigorously that the suitably scaled stochastic process converges in law in Skorokhod space to SKT system \eqref{eq:SKT:cons}  and we perform this  space and time scaling limit at once. Besides, individuals of each species move independently with a rate proportional to the number of individuals of the other species,  on the same site. We do not need to make them move by pair,  which may be hard to justify regarding phenomenon at stake. Indeed, we do not need a  reversibility property and do not use the entropic structure. The main difficulty to prove convergence of the stochastic process at once lies in the control of the cumulative  quadratic rates due to local interactions when the number of sites becomes large.
As far as we have seen,  entropy structure does not provide the suitable control of these non-linear terms and a way  to get tightness and identification in general. We use a different approach  
 based on generalized duality. This provides   quantitative estimates  in terms of space discretization and size of population. Moreover, 
at the level of the PDE system, it implies a local uniqueness result for bounded solutions of the SKT system. The duality approach allows to compare locally the stochastic process with its semi-discrete deterministic approximation. It is optimal in the sense that it provides the good  time space scaling  for such an approximation.  


Let us describe now the stochastic individual-based model. The population is spatially distributed among $M$ sites. The process under consideration is a  continuous time Markov chain $(\bm{U}(t),\bm{V}(t))_{t\geq0}$ taking values in $\NN^M\times\NN^M$. The two coordinates count the number of individuals of each species at each site, for each time $t\geq 0$.
Each individual of each species follows a random walk and its jumps rate increases linearly with respect to the number of individuals of the other species. The dynamic is defined by the jump rates as follows.
For any vector of configurations $(\bm{u},\bm{v})\in\NN^M\times\NN^M$, the transitions are
\begin{align*}
  \begin{array}{lcl}
     \bm{u} \mapsto \bm{u} + \bigl( \bm{\mathrm{e}}_{i+\theta} - \bm{\mathrm{e}}_{i} \bigr) & \text{ at rate } & 2u_i(d_1+a_{12} v_i),  \\
     \bm{v} \mapsto \bm{v} \,+ \bigl( \bm{\mathrm{e}}_{i+\theta} - \bm{\mathrm{e}}_{i} \bigr) & \text{ at rate } & 2v_i(d_2+a_{21}u_i), 
  \end{array}
\end{align*}
where $(\bm{\mathrm{e}}_j)_{1\leq j\leq M}$ is the canonical basis of $\RR^M$, $\bm{\mathrm{e}}_0=\bm{\mathrm{e}}_{M}$, $\bm{\mathrm{e}}_{M+1}=\bm{\mathrm{e}}_1$ and $\theta\in \{-1,1\}$ with both values equally likely.
 Let us mention that hydrodynamic  limits of  other stochastic models with repulsive  species have been considered, in particular in the context of exclusion processes, see e.g. \cite{Quastel}. In that case, local densities  are bounded so difficulties and limits are different. In an other direction, stochastic versions of the limiting SKT systems have been considered, see e.g. \cite{DHJ2020}. We also mention \cite{HJV} for hydrodynamic limit to fast diffusion, where the non linearly is also in the motion component. The model is different and we are interested here in  the interaction of two species, without self diffusion. Besides our techiques are different since we do not rely and do not need  an entropic structure and the control of the approximation involves a different distance.

This work contains two main results which at first sight can appear unrelated in their formulation. 
The first result is a quantitative stability estimate on the SKT system which bounds the distance between two solutions in terms of their initial distance. This result is based on a new  duality lemma and applies for bounded solutions, only if one of them is small enough. As a by-product of this stability estimate, we prove uniqueness of bounded solutions of the conservative SKT system under a smallness condition which does not imply ellipticity for the system (later on, we will comment on this non-perturbative smallness condition).  This result 
is valid in arbitrary dimension and is, as far as our knowledge goes, new. Uniqueness theorems for (only) bounded solutions of the full SKT system are missing in the current literature \cite{chen_jun_unique,chen_jun_unique_bis,pham_temam_unique}.

The second main result is the convergence of the properly scaled sequence of processes $(\bm{U}^{M,N},\bm{V}^{M,N})_{M,N\in\NN}$  to the SKT system. We obtain quantitative estimates of the gap between the trajectories of this process extended to the continuous space and the solution of SKT system, in a large population  and diffusive regime. This analysis is performed in a one dimensional setting for the space variable. The strategy is to insert the semi-discrete model proposed in \cite{DDD2019} and estimate separately the gap between our stochastic process and this semi-discrete system 
and then, estimate (with enough uniformity) the distance between the semi-discrete system and the continuous  SKT limit.  Following this plan, we  first  propose a general estimate, which relies on naive bounds of the quadratic diffusion term. Roughly, we  first bound locally the size of the population by the (constant) total number of individuals.   These  bounds allow for convergence with a fixed number of sites but lead to an unreasonable assumption of a superexponential number of individuals per site when the number of sites increases. When we faced this difficulty, we tried to obtain an estimate as sharp as possible to capture the good scales and compare on each site the different objects. It's during this step that we discovered the stability estimate described above, which is interesting for its own sake. A  nice feature of this stability estimate is that we can transfer it onto the semi-discrete  and stochastic setting.   We obtain then the convergence of the stochastic model towards the SKT system, with sharp estimates and relevant size scales. This asymptotic study shares a similar limitation as the previous paragraph: it holds only under the assumption of small regular solution of the SKT system, which is ensured by Amann's theorem \cite{amann90,amann90b}.

The paper is organized as follows. In the end of this section,  we collect several notations which will be used throughout the paper. In Section~\ref{sec:procmain} we define the sequence of stochastic processes we consider and we recover the semi-discrete system introduced in \cite{DDD2019}. We also  state our two main results and comment on some potential extensions. In Section~\ref{sec:compactness} we show the convergence in law in path space of the stochastic process towards the semi-discrete system when the number of individuals goes to infinity but the number of sites remains fixed. We provide a quantification of this convergence. It
implies the general (no restriction on the limiting SKT system) but naive (in terms of scales) convergence  discussed above. Then, Section~\ref{sec:duality} is dedicated to the duality estimates with source terms and their consequences. These duality estimates account for the interacting system when one of the population is seen as an exogenous environment, which amounts to decouple the two species.  In a first short paragraph (Subsection~\ref{subsec:contset}) we state and prove the generalized duality lemma and its application to the stability estimate of the SKT system in the continuous setting. This paragraph is the only one of the study in which we work in arbitrary dimension for the space variable. Then, the rest of Section~\ref{sec:duality} focuses on the translation of these estimates in the semi-discrete setting. This includes the definition of reconstruction operators, the study of the discrete laplacian matrix and the translation of classical function spaces into the discrete setting. Eventually in Section~\ref{sec:quantification}, we apply the previous machinery to the difference between the stochastic process and the approximated system that solutions of \eqref{eq:SKT:cons}  solve when looked at a semi-discrete level. We then deduce our main asymptotic  theorem by controlling the martingales and approximation terms. 
In a short appendix, we also give a dictionary which gives the correspondence of different objects in the discrete and continuous settings.

\subsection{Notation}\label{subsec:not}

\subsubsection*{Finite-dimensional vectors}
Throughout the article, vectors will always be written in bold letters. The canonical basis of $\RR^M$ will be denoted $(\bm{\mathrm{e}}_j)_{1\leq j\leq M}$. Due to the periodic boundary condition that we will use, we will frequently use the convention $\bm{\mathrm{e}}_0=\bm{\mathrm{e}}_M$ and $\bm{\mathrm{e}}_{M+1}=\bm{\mathrm{e}}_1$.

Given $M\in\NN$ and $p\in[1,\infty)$ we introduce a rescaled norm $\|\cdot\|_{p,M}$ defined for $\bm{x}\in\RR^M$ by $\|\bm{x}\|_{p,M}:=M^{-1/p}\|\bm{x}\|_p$
where $\|\cdot\|_p$ denotes the usual $\ell^p$ norm on $\RR^M$. Similarly, we define the rescaled euclidean inner-product $(\cdot|\cdot)_M$ of $\RR^M$ for $\bm{x},\bm{y}\in\RR^M$ by 
$(\bm{x} |\bm{y})_M=M^{-1}(\bm{x}|\bm{y})$,
where $(\cdot|\cdot)$ is the usual inner-product of $\RR^M$ so that $\|\bm{x}\|_{2,M}^2 = (\bm{x}|\bm{x})_M$.

The symbol $\odot$ is the internal Hadamard product on $\RR^M$, that is $(\bm{x}\odot \bm{y})_i = x_i y_i$. We will also often use (when it makes sense) the operator $\bm{x}\oslash \bm{y}$ defined by $(\bm{x}\oslash \bm{y})_i=x_i/y_i$ and the ``vectorial'' square-root $\bm{x}^{1/2}$ whose components are $(\sqrt{x_i})_{1\leq i\leq M}$.

The arithmetic average of all the components of a vector $\bm{x}$ will be denoted $[\bm{x}]_M \coloneqq M^{-1}\sum_{i=1}^M x_i$.

The vector of $\RR^M$ for which every component equals $1$ is denoted $\mathbf{1}_M$. The orthogonal projection onto $\mathrm{Span}_\RR(\mathbf{1}_M)^\perp$ is denoted with a tilde, that is: $\widetilde{\bm{x}}= \bm{x}-\meanM{\bm{x}} \mathbf{1}_M$.

Finally, for $\bm{x},\bm{y}\in\RR^M$ we write $\bm{x}\geq \bm{y}$ whenever $\bm{x}-\bm{y}\in\RR_+^M$.

\subsubsection*{Functions}

We will manipulate random and deterministic functions which may depend on the time variable $t\in\RR_+$ and the space variable $x\in\T^d$, where $\T\coloneqq\RR/\ZZ$ is the flat periodic torus. 
We will rely on the following convention for functions: uppercase letters will be reserved for random elements whereas lowercase letters will represent deterministic functions. Accordingly to the previous paragraph, vector valued functions will be denoted in bold whereas scalar valued functions will be denoted with the normal font.

Quite often results will be stated on a fixed time interval $[0,T]$. For this reason, we introduce the periodic cylinder $Q_T \coloneqq [0,T] \times\T^d$. For any function space $E$ defined on $\T^d$ or $Q_T$, the corresponding norm will be denoted $\|\cdot\|_E$, \emph{e.g.} $\|\cdot\|_{L^2(\T^d)}$. In case of a Hilbert structure, the inner-product will be denoted by $(\cdot|\cdot)_E$, \emph{e.g.} $(\cdot|\cdot)_{L^2(\T^d)}$. We will frequently use the $H^s(\T^d)$ Sobolev space and their homogeneous subspace $\dot{H}^s(\T^d)$ constituted of those elements having a vanishing average on $\T^d$.  

\medskip

For random functions $Z:\Omega\times Q_T \rightarrow \mathbb{R}$ we will frequently use the norm 
\begin{align}
  \label{eq:tripnorm}\vertiii{Z}_T \coloneqq \left(\sup_{t\in[0,T]} \EE\big(\|Z(t)\|_{H^{-1}(\T^d))}^2\big) \, + \, \EE\big(\|Z\|_{ L^2(Q_T)}^2\big) \right)^{1/2}.
  \end{align}
  Note that in the case of a deterministic function $z$, the previous norm becomes simply \begin{align}
  \label{eq:tripnorm:det}
\vertiii{z}_T \coloneqq \left(\|z\|_{L^\infty([0,T];H^{-1}(\T^d))}^2+ \|z\|_{ L^2(Q_T)}^2 \right)^{1/2}.
  \end{align}

Finally, for any metric space $X$, $D([0,T],X)$ denotes the space of c\`adl\`ag functions from $[0,T]$ to $X$ endowed with the Skorokhod topology.

\section{Main objects and results}\label{sec:procmain}

Before stating our main results, we need to define precisely the objects that we aim at considering.


\subsection{Repulsive random walks and scaling}


Let us define the stochastic process by means of a trajectorial representation using Poisson random measures. We consider a probability space $(\Omega, \mathcal{F}, \PP)$ and introduce a family of independent Poisson random measure $(\mathcal{N}^j)_{j\in\NN}$ on $\RR_+\times\RR_+\times\lbrace-1,1\rbrace$ with common intensity $\mathrm{d}s\otimes \dd\rho \otimes \beta(\mathrm{d}\theta)$, where $\beta$ is the law of a Bernoulli$\bigl(\frac{1}{2}\bigr)$ random variable. We refer to 
Definition 8.1, Chapter 1 in \cite{IW} for the definition of Poisson random measure. Moreover,
 the initial data $(\bm{U}(0),  \bm{V}(0))$ almost surely belongs to $\NN^M\times\NN^M$. The corresponding process $(\bm{U}(t),\bm{V}(t))_{t\geq0}$ is then defined as the unique strong solution in $D([0,\infty), \mathbb N^{2M})$ of the following system of stochastic differential equations (SDEs) driven by the aforementioned measures
\begin{equation*}
  \left\{
  \begin{aligned}
    \bm{U}(t) ={}& \bm{U}(0) + \int_0^t \int_{\RR_+ \times \{-1,1\}} \, \sum_{j=1}^M \1_{\rho \leq 2U_j(s^{-})(d_1 + a_{12} V_j(s^{-}))} \bigl( \bm{\mathrm{e}}_{j+\theta} - \bm{\mathrm{e}}_{j} \bigr) \,\mathcal{N}^j(\dd s, \dd \rho, \dd \theta), \\
    \bm{V}(t) ={}& \bm{V}(0) + \int_0^t\int_{\RR_+ \times \{-1,1\}} \, \sum_{j=1}^M \1_{\rho \leq 2V_j(s^{-})(d_2 + a_{21}U_j(s^{-}) )} \bigl( \bm{\mathrm{e}}_{j+\theta} - \bm{\mathrm{e}}_{j} \bigr) \,\mathcal{N}^j(\dd s,  \dd \rho, \dd \theta),
  \end{aligned}
  \right.
\end{equation*}
where the jump rates $d_1,d_2,a_{12}$ and $a_{21}$ are the one of \eqref{eq:SKT:cons}. Let us first explain roughly the terms of these SDEs. The  measures $\mathcal N$ produce  the sources of randomness for the jumps; the indicator functions ${\bf 1}$ select the jumps which actually occur depending on the number of individuals of each species;   the jump of one individual from $j$ to its neighbor $j+\theta$ induces the variation
$\bm{\mathrm{e}}_{j+\theta} - \bm{\mathrm{e}}_{j}$ on the vector couting the population size of site, resp.  $\bm{U}$ and $\bm{V}$. The  existence and uniqueness of this system of SDEs can be proved by  induction using the fact that the process is constant between two jumps and the total jump rate is bounded.  Indeed, the  total population size of each species
is constant along time:  $\| \bm{U}(t)\|_{1,M}=\| \bm{U}(0)\|_{1,M}$, $\| \bm{V}(t)\|_{1,M}=\| \bm{V}(0)\|_{1,M}$. Therefore, 
conditionally on the initial value $(\bm{U}(0),  \bm{V}(0))$, the 
process $(\bm{U}(t),\bm{V}(t))_{t\geq0}$ is a pure jump Markov process on a finite state space 
with bounded rates.
The strong uniqueness and existence of this system of SDEs are actually also a consequence of  more general statements for SDEs with jumps described a Poisson random measure, see in particular Theorem 9.1, Chapter 4 in \cite{IW}.
  


We are interested in the approximation  (hydrodynamic limit) when the population size and the number of sites tend to infinity. Informally,  we consider 
$$(\bm{U}^{M,N}(t),\bm{V}^{M,N}(t))_{t\geq0}=(\bm{U}(M^2t)/N,\bm{V}(M^2t)/N)_{t\geq0}$$ 
but now  interaction  occurs through the local density of individuals.  The scaling parameter $N\in\NN^*$ 
yields  the normalization of the population per site and provides a limiting density when $N$ goes to infinity. The initial population per site is of order of magnitude $N$ and each species' motion rate is an affine function
of
 the density of the other species on the same site.
The motion of each individual is centered and we consider the diffusive regime, which leads the time acceleration  term by a factor $M^2$. 
This time acceleration is equivalent
to multiply the jump rates by the same factor. 

More precisely,  for $i,j=1,2$ and $t\geq 0$, we set 
\begin{align*}
  \eta_{1,j}^{M,N}(t) &\coloneqq 2M^2NU_j^{M,N}(t)\bigl(d_1 + a_{12}V_j^{M,N}(t)\bigr), \\
  \eta_{2,j}^{M,N}(t) &\coloneqq 2M^2NV_j^{M,N}(t)\bigl(d_2 + a_{21}U_j^{M,N}(t)\bigr).
\end{align*}
Given an initial data $(\bm{U}^{M,N}(0) , \bm{V}^{M,N}(0))$, the normalized process $(\bm{U}^{M,N}(t),\bm{V}^{M,N}(t))_{t\geq0}$ is defined as the unique strong solution in $D([0,\infty),\RR_+^{2M})$ of the following system of SDEs 
\begin{equation}\label{eq:traj}
  \left\{
  \begin{aligned}
    \bm{U}^{M,N}(t) &= \bm{U}^{M,N}(0) + \int_0^t\int_{\RR_+ \times \{-1,1\}} \sum_{j=1}^M\1_{\rho \leq \eta_{1,j}^{M,N}(s^-)} \frac{\bm{\mathrm{e}}_{j + \theta} - \bm{\mathrm{e}}_j}{N} \,\mathcal{N}^j(\dd s, \dd \rho, \dd \theta), \\
    \bm{V}^{M,N}(t) &= \bm{V}^{M,N}(0) + \int_0^t\int_{\RR_+ \times \{-1,1\}} \sum_{j=1}^M\1_{\rho \leq \eta_{2,j}^{M,N}(s^-)} \frac{\bm{\mathrm{e}}_{j + \theta} - \bm{\mathrm{e}}_j}{N} \,\mathcal{N}^j(\dd s, \dd \rho, \dd \theta).
  \end{aligned}
  \right.
\end{equation}

\subsection{The intermediate (semi-discrete) system}
To estimate the gap between the discrete stochastic process \eqref{eq:traj} and  the SKT system \eqref{eq:SKT:cons}, we are going to use a third system on which our asymptotic analysis will pivot
\begin{equation}
  \label{eq:SKTsemidis}
  \left\{                                      
  \begin{aligned}
    \frac{\mathrm{d}}{\mathrm{d} t} 
    \bm{u}^M(t) -\Delta_M ( d_1\bm{u}^M(t) + a_{12}\bm{u}^M(t)\odot \bm{v}^M(t) \bigr) &= 0, \\
    \vphantom{\int_0^t} \frac{\mathrm{d}}{\mathrm{d} t} \bm{v}^M(t) - \Delta_{M}\bigl( d_2\bm{v}^M(t) + a_{21}\bm{u}^M(t)\odot \bm{v}^M(t)\bigr) &= 0,
  \end{aligned}
  \right.                                
\end{equation}
where the unknowns are the vector valued curves $\bm{u}^M,\bm{v}^M\colon \RR_+\to\RR^M$, and the matrix $\Delta_M$ is the periodic laplacian matrix, that is
  \begin{equation}
    \label{eq:DeltaM}
  \Delta_M\coloneqq M^2
  \begin{pmatrix}
    -2 & 1 & 0 &\cdots& 1 \\
    1 & -2 & 1 &\cdots& 0\\
    \vdots & \ddots &\ddots & \ddots &\vdots \\
    0 & \cdots & 1 & -2 & 1 \\
    1 & \cdots & 0 & 1 & -2
  \end{pmatrix}\in\textnormal{M}_M(\RR).
\end{equation}
This semi-discrete system corresponds to a large population approximation but fixed number of sites $M$.
Existence and uniqueness for \eqref{eq:SKTsemidis} can be proven using the standard Picard-Lindelöf theorem, as this is done in \cite{DDD2019} where this semi-discrete system has been introduced. 

\subsection{Statements}





Our first main result is a stability estimate for the conservative SKT system \eqref{eq:SKT:cons}. As far as our knowledge goes, this result is  new in the context of weak solutions for the SKT system. To measure the distance between two solutions on a time interval $[0,T]$, we use the norm defined in \eqref{eq:tripnorm:det} in the deterministic setting, see Section~\ref{subsec:not}. We define also the affine functions $\mu_i\colon\RR\to\RR$ for $i=1,2$, by $\mu_i(x)\coloneqq d_i+a_{ij}x$ with $\{i,j\}=\{1,2\}$.
\begin{thm}\label{thm:sktstab}
Let $T>0$ and consider a couple $(u,v)\in L^\infty(Q_T)^2$ and $(\overline{u},\overline{v})\in L^\infty(Q_T)^2$ of non-negative bounded weak solutions of the SKT system \eqref{eq:SKT:cons}, respectively initialized by $(u_0,v_0)\in L^\infty(\T^d)^2$ and $(\overline{u}_0,\overline{v}_0)\in L^\infty(\T^d)^2$.  If the following condition 
 \begin{align}\label{ineq:small}
   \|u\|_{ L^\infty(Q_T)}\|v\|_{ L^\infty(Q_T)} <  \frac{d_1 d_2}{a_{12}a_{21}},
   \end{align}
is satisfied, then we have the stability estimate
\begin{equation*}
  \begin{split}
  \vertiii{u-\overline{u}}_T^2+\vertiii{v-\overline{v}}_T^2 &\lesssim
  \|u_0-\overline{u}_0\|_{H^{-1}(\T^d)}^2 + \|v_0-\overline{v}_0\|_{H^{-1}(\T^d)}^2 \\
  &\qquad + T\Big([u_0-\overline{u}_0]_{\T^d}^2 \| \mu_1(\overline{v}_0)\|_{ L^1(\T^d)}  + [v_0-\overline{v}_0]_{\T^d}^2 \|\mu_2(\overline{u}_0)\|_{ L^1(\T^d)}\Big),
  \end{split}
\end{equation*}
where the constant behind $\lesssim$ depends only on $a_{ij},d_i,\|u\|_{L^\infty(Q_T)},\|v\|_{L^\infty(Q_T)}$, and $\vertiii{\cdot}_T$ is defined by \eqref{eq:tripnorm:det}. In particular, if a bounded non-negative solution satisfies \eqref{ineq:small} then, there is no other bounded non-negative solution sharing the same initial data. 
\end{thm}
\begin{rmk}
In case of equality in the smallness condition \eqref{ineq:small}, uniqueness remains but the stability estimate controls only the $H^{-1}$ part of the $\vertiii{\cdot}_T$ norm. 
                                            \end{rmk}
The proof of Theorem~\ref{thm:sktstab} relies on a generalized duality lemma presented in Subsection~\ref{subsec:contset} and on the concept of \emph{dual solutions} developed in \cite{moussa_nonloc_tri}, for the Kolmogorov equation. The uniqueness result contained in Theorem~\ref{thm:sktstab} is conditional: \emph{if} there exists a bounded (non-negative) solution $(\overline{u},\overline{v})$ satisfying \eqref{ineq:small}, then it is unique in the class of bounded weak solutions. The existence of \emph{global} bounded solutions for the SKT system is a long standing challenge in the context of cross-diffusion systems. Partial results are known, in the wake of the quest of even more regular solutions (which are in particular bounded), like \cite{HoanNguPha} or \cite{louwin} that we already cited. In the weak solutions setting, the paper \cite{jungzamp} gives sufficient --yet restrictive-- conditions on the coefficients of the SKT system to ensure boundedness. Since the previous results are rather constraining on the coefficients, we prefer to rely on Amann's theory \cite{amann90,amann90b} and understand Theorem~\ref{thm:sktstab} as a \emph{local} result which holds for initial data satisfying \eqref{ineq:small}. However, we emphasize that the condition \eqref{ineq:small} is considerably less restrictive that the standard perturbative assumptions considered for cross-diffusion systems and we call it for this reason a \emph{non-perturbative smallness condition}. This condition does not apply to both species but only on the product of the densities: one of the two functions $u$ and $v$ can be huge. Secondly, the stability and uniqueness result contained in Theorem~\ref{thm:sktstab} is not of a "weak-strong" type: both solutions are weak (only bounded, no \emph{a priori} assumptions on the spatial derivatives) which is, as far as our knowledge goes, a substantial step in the analysis of cross-diffusion systems. Indeed, because of the stiffness of those systems a common strategy to recover a well-posedness result is to impose on the coefficients or the solution itself a constraint ensuring that the total system is uniformly elliptic in the sense that it can be written as $\partial_t U - \textnormal{div}(A(U)\nabla U) =0$ with a diffusion matrix $A(U)$ uniformly positive, that is satisfying $\langle A(U)X,X\rangle\gtrsim\|X\|_2^2$ pointwisely for $X\in\RR^2$. In our case a direct computation shows that the matrix $A(U) = A(u,v)$ is \[A(u,v)=\begin{pmatrix} d_1 +a_{12}v & a_{21}u\\a_{12}v & d_2+a_{21}u\end{pmatrix}.\] For non-negative densities $u$ and $v$, the trace of the previous matrix field is positive, so its positiveness (as a quadratic form) is equivalent to $\det(A(U)+A(U)^T)>0$, that is $4(d_1+a_{12}v)(d_2+a_{21}u) \geq (a_{21}u  + a_{12}v)^2$. Since this inequality is trivially true for $(u,v)=0$, the previous computation paves the way to well-posedness results for small enough densities or strong enough self-diffusion w.r.t. the cross-diffusion coefficients (see \emph{e.g.} \cite{berbur,pham_temam_unique,chen_jun_unique_bis}). In all these results, the setting in which the solution are built is in fact strongly elliptic and in the best case, a weak-strong uniqueness result is obtained (see \emph{e.g.} \cite{berbur}). In our case, the condition \eqref{ineq:small} does not ensure strong ellipticity for the system: $v$ could be very small and $u$ very large and still we could have $4(d_1+a_{12}v)(d_2+a_{21} u) < (a_{21} u  + a_{12} v)^2$. In particular, our stability result is of weak-weak type. As the proof of Theorem~\ref{thm:sktstab} (which is done in Subsection~\ref{subsec:contset}) is totally insensitive to the dimension $d$, it is here stated in full generality. However, the remaining part of the paper is sensitive to the dimension and   will focus on the case $d=1$.
It deals with the approximation of the SKT system by stochastic processes.\\

Before stating our second main result, let us comment briefly  Section~\ref{sec:compactness} in which we propose a first estimates of the gap between the stochastic process defined by \eqref{eq:traj} and the semi-discrete system \eqref{eq:SKTsemidis} on a fixed interval $[0,T]$. The methodology at stake in this paragraph, which is quite rough, allows for asymptotic quadratic closeness between these two objects, \emph{provided that}, as $N,M \rightarrow +\infty$, we have the following 
\begin{equation}
\label{scales}
N \gg M^4 \exp(cM^4T),
\end{equation}
where $c$ is some  constant which will become more explicit in the next section. Combining this fact with the compactness result \cite[Theorem 8]{DDD2019}, we obtain convergence (up to a subsequence) of our stochastic process towards a weak solution of the SKT system. These estimates and convergence yield first results which are    general  in terms of parameters and form of the solution.
However, the limitations of  this approach are twofold. First, the scaling condition \eqref{scales} involves a superexponential and time dependent number of individuals per site in order to make  to be able to sum local estimates. As we will see, and as we can guess from the form of quadratic variations, this scaling is too restrictive for convergence.
Second, this approach necessitates a self-diffusion term in the limiting system  in order to use the compactness result of \cite{DDD2019}. Indeed, self diffusion term tends  to regularize the solution.

We develop then a different approach, based on the discrete translation of Theorem~\ref{thm:sktstab}. This alternative method does not rely on \cite{DDD2019}, so that self-diffusion is not needed in the system. The convergence result is obtained by means of a quantitative estimate which bounds the expectation of the $\vertiii{\cdot}_T$-norm
of the gap  between the stochastic processes and the solution of the SKT system. In particular, there are no compactness tools used and the entropy of the system is not needed. 
Convergence is then guaranteed only with a quadratic number of individuals per site. This corresponds to the expected scaling for having local control of the stochastic process by its semi-discrete approximation, since beyond this scaling quadratic variations do not vanish. The main disadvantage of this new method is that, like for Theorem~\ref{thm:sktstab}, it needs the existence of a bounded solution satisfying condition \eqref{ineq:small}.

In order to state the following result, we need to introduce, for any integer $M\geq 1$, the discretization of the flat (one dimensional) torus $\T$
\begin{align}\label{eq:discT}
\T_M\coloneqq\{x_1,x_2,\cdots,x_M\}, \quad \text{with $x_k =\frac{k}{M}$, for $1\leq k\leq M$}.
\end{align}
Given a vector $\bm{u}\in\RR^M$, classically there exists exactly one continuous piecewise linear function defined on $\T$ for which its value on each point $x_k$ of $\T_M$ is given by $u_k$; we denote this function $\pi_M(\bm{u})$. We adapt the same notation if instead of $\bm{u}$ one considers a vector valued map $\bm{U}$ (which could depend on the event $\omega$ or the time $t$ for instance), so that $\pi_M(\bm{U})$ becomes a real-valued map.\\
This time, to measure the distance between those random functions we use the probabilistic version of the distance introduced in Subsection~\ref{subsec:not}, that is \eqref{eq:tripnorm}.
\begin{thm}\label{thm:quantitativeresult}
  Let $T>0$. In the one dimensional case $d=1$, assume the existence of a non-negative solution $u,v$ belonging to $L^\infty(Q_T)\cap L^2(0,T;H^3(\T))$ of the system \eqref{eq:SKT:cons}, initialized by $u_0,v_0$ in $L^\infty\cap H^3(\T)$ and satisfying the assumption \eqref{ineq:small}. Consider the stochastic processes $(\bm{U}^{M,N},\bm{V}^{M,N})$ defined by \eqref{eq:traj} and assume the existence of $\C_0$ such that   for all $M,N\in\NN$,
  \begin{align}
    \label{condinitborn}
    \|\bm{U}^{M,N}(0)\|_{1,M}+\|\bm{V}^{M,N}(0)\|_{1,M} \leq \C_0, \quad \text{almost surely.}
  \end{align}
 Then, there exists a sequence $(\delta_M)_M\in\mathbb{R}_{>0}^{\mathbb{N}}$ converging to $0$ and a constant $\textnormal{D}>0$ such that for any $(M,N)\in \NN^2$ satisfying $N \geq M^2 \textnormal{D}$, there holds
  \begin{align}
    \label{ineq:ibelieveicancry}  
    &\vertiii{ \pi_M\bigl(\bm{U}^{M,N}\bigr) - u}_{T}^2+\vertiii{ \pi_M\bigl(\bm{V}^{M,N}\bigr) - v}_{T}^2\\
    &\ \, \lesssim  \EE\Big[\|\pi_M\bigl(\bm{U}^{M,N}(0)\bigr) - u_0\|_{H^{-1}(\T)}^2+\|\pi_M\bigl(\bm{V}^{M,N}(0)\bigr) - v_0\|_{H^{-1}(\T)}^2\Big]   + \delta_M+ \frac{M^2}{N}, \notag 
  \end{align}
  where $\vertiii{\cdot}_T$ is defined by \eqref{eq:tripnorm} and 
  the symbol $\lesssim$ and the constant $\textnormal{D}$ depend (only) on $\C_0,T$, $d_i,a_{ij},\|u\|_{L^\infty(Q_T)},\|v\|_{L^\infty(Q_T)}$, while the sequence $(\delta_M)_M$ depends only on the solution $u,v$.
\end{thm}
\begin{rmk}
If the solution $u,v$ is assumed to be more regular, the convergence of $(\delta_M)_M$ can be estimated more accurately. See Remark~\ref{rem:delta} for more details. Also, $L^2(0,T;H^3(\T))$ is not optimal and could be replaced  by $L^2(0,T;H^{2+s}(\T))$ for any $s>1/2$.
\end{rmk}
This immediately implies the following convergence for the $\vertiii{\cdot}_T$-norm.
\begin{cor}
Let $T>0$. Under the assumptions of Theorem \ref{thm:quantitativeresult}, consider an extraction function $\phi\colon\NN\to\NN$ such that $M^2=\textnormal{o}(\phi(M))$.  If the initial positions of the individuals are well-prepared in the sense that
\begin{align*}
\EE\Big[\|\pi_M\bigl(\bm{U}^{M,\phi(M)}(0)\bigr) - u_0\|_{H^{-1}(\T)}^2+\|\pi_M\bigl(\bm{V}^{M,\phi(M)}(0)\bigr)-v_0\|_{H^{-1}(\T)}^2\Big] \operatorname*{\longrightarrow}_{M\rightarrow +\infty} 0,
\end{align*}
then we have 
\[
  \lim_{M\rightarrow \infty} \vertiii{ \pi_M\bigl(\bm{U}^{M,\phi(M)}\bigr) - u}_{T}^2+\vertiii{ \pi_M\bigl(\bm{V}^{M,\phi(M)}\bigr) - v}_{T}^2=0,
\]
where $\vertiii{\cdot}_T$ is defined by \eqref{eq:tripnorm}.
  \end{cor}

\medskip
 
 Let's end up with other perspectives  and extensions we have in mind. 
 
 \medskip
 
We considered in this work periodic boundary conditions since the domain of study is the flat torus (be it in dimension $1$ or more). For Theorem~\ref{thm:sktstab}, our method of proof relies on fine energy estimates involving negative Sobolev and quadratic norms.  There is no doubt that the method of proof we introduce can be adapted without much difficulties to boundary conditions that are more frequently used in the population dynamics (as homogeneous Dirichlet or Neumann for instance). For the description of the stochastic individual based model and Theorem~\ref{thm:quantitativeresult}, it amounts to kill the individuals hitting the boundary (for homogenous Dirichlet boundary condition) or reflecting the motion by authorizing only jumps that remain inside the domain (in the case of Neumann boundary condition). This would be an interesting extension of our work. It would lead to additional  technical difficulties, but we do not see any major issue in the application of our method.

\medskip

The approach we have developped in this work differs from  more classical techniques relying on reversibility property or the existence of suitable Lyapunov functional. 
This point of view allows to get convergence in a strong sense. This ensures that the number of individuals of the stochastic process on a given site is well approximated by the limiting SKT system. Moreover, we expect that this approach can be extended in several directions and could be use for more sophisticated models.
An extension for which we are rather confident is the generalization of our asymptotic analysis to higher dimension. An upper limit  is fixed by the avatar of the Bramble-Hilbert lemma, which is Lemma \ref{lem:bh}. This latter  demands a Sobolev embedding $H^2(\T^d) \hookrightarrow \mathscr{C}^0(\T^d)$, which holds only for $d=1,2,3$. On the other hand, keeping in mind that solutions of the system of PDEs represent a population density in an environment, the exploration of such system in dimensions greater than 4 loses some interest.  We thus believe that the analysis that we develop is  adaptable to dimensions 2 and 3. However, this seems to imply a technical cost. \\
Besides, we believe that birth or death of individuals can be included in our framework. This is relevant for modeling purposes and would add a reaction term in the limiting system. Originally the SKT system was introduced because of its ability to produce segregated states. But these particular equilibria result from the interaction of the cross-diffusion rates \emph{and} the reaction rates (that we have chosen to neglect here) terms. \\
 Finally, we expect that our proofs can be also extended to
 more general  cross-diffusion terms  or self-diffusion.
In a nutshell, we believe that the main lines of our approach should work for various extensions and could lead to interesting future works.


\section{A first approach}\label{sec:compactness}

The trajectorial representation \eqref{eq:traj}  yields for each coordinate of $\bm{U}^{M,N}$
\begin{align}\label{eq:trajectory}
  U_i^{M,N}(t) &= U_i^{M,N}(0) - \frac{1}{N}\int_0^t \int_{\RR_+ \times \{-1,1\}} \1_{\rho\leq \eta_{1,i}^{M,N}(s^-)}\,\mathcal{N}^i(\dd s,\dd\rho, \dd\theta) \notag \\
  &\qquad +\frac{1}{N}\int_0^t \int_{\RR_+ \times \{-1,1\}} \1_{\rho\leq \eta_{1,i-1}^{M,N}(s^-)}\1_{\theta=1}\,\mathcal{N}^{i-1}(\dd s,\dd\rho, \dd\theta) \notag \\ 
  &\qquad\qquad + \frac{1}{N}\int_0^t \int_{\RR_+ \times \{-1,1\}} \1_{\rho\leq \eta_{1,i+1}^{M,N}(s^-)}\1_{\theta=-1} \,\mathcal{N}^{i+1}(\dd s,\dd\rho, \dd\theta).
\end{align}


By compensating the 
 Poisson random measure, we obtain the following semimartingale decomposition (see Definition 4.1 in Chapter 2 in \cite{IW})
\begin{align}\label{eq:semimg_u}
  \bm{U}^{M,N}(t) &=\bm{A}^{M,N}(t)+\bm{\mathcal M}^{M,N}(t),
\end{align}
where $\bm{A}^{M,N}=(A^{M,N}_i)_{1\leq i\leq M}$ is a  continuous process defined by
\[
  \bm{A}^{M,N}(t)= \bm{U}^{M,N}(0) + \int_0^t d_1\Delta_M\bm{U}^{M,N}(s) \,\dd s + \int_0^t a_{12}\Delta_{M}\bigl(\bm{U}^{M,N}(s) \odot \bm{V}^{M,N}(s)\bigr) \,\dd s,
\]
with $\Delta_M$ as defined in \eqref{eq:DeltaM}, and $\bm{\mathcal M}^{M,N}=(M_i^{M,N})_{1\leq i\leq M}$ is a martingale. More precisely, for any $1\leq i\leq M$, $\mathcal M_i^{M,N}$ is a square integrable martingale whose predictable quadratic variation is given for $t\geq 0$ by
\begin{align}\label{eq:predquad}
  \bigl\langle \mathcal M_i^{M,N} \bigr\rangle(t) &= \frac{M^2}{N}\int_0^td_1\Bigl(2U_i^{M,N}(s)+U_{i+1}^{M,N}(s)+U_{i-1}^{M,N}(s)\Bigr)\dd s \\
  & \hspace{-1.2em} + \frac{M^2}{N}\int_0^ta_{12}\Bigl(2U_i^{M,N}(s)V_i^{M,N}(s)+U_{i+1}^{M,N}(s)V_{i+1}^{M,N}(s) + U_{i-1}^{M,N}(s)V_{i-1}^{M,N}(s)\Bigr)\dd s.
  \nonumber
  \end{align}
  Then, there exists a constant $C$, which only depends on the diffusion coefficients, such that
\begin{align}  
\sum_{i=1}^M \bigl\langle \mathcal M_i^{M,N} \bigr\rangle(t)  &\leq  C\,  \frac{M^2}{N} \, \int_0^t \Bigl( \| \bm{U}^{M,N}(s)\|_1 + \|\bm{U}^{M,N}(s)\|_2^2 + \| \bm{V}^{M,N}(s)\|_2^2 \Bigr) \, \dd s. \label{bornVQ}
\end{align}
We refer to Chapters  1.6  in \cite{IW} for   the  definitions of martingales and to Chapter 2.2  for the more specific form of martingales appearing here.
The analogous decomposition holds for  the coordinates of $(\bm{V}^{M,N}(t))_{t\geq0}$, the second species.

Let us give first estimates of the gap between the stochastic process and its approximation in large population  for a fixed number of sites. Let
\[
  \bm{\mathcal U}^{M,N}(t)=\bm{U}^{M,N}(t) - \bm{u}^M(t), \quad \bm{\mathcal V}^{M,N}(t)=\bm{V}^{M,N}(t) - \bm{v}^M(t).
\]
\begin{prop}\label{prop:rateL2} We assume that there exists $C_0>0$ such that almost surely, for any $M,N\geq 1$,
$$\|\bm{U}^{M,N}(0)\|_{1,M}+\|\bm{V}^{M,N}(0)\|_{1,M}+ \|\bm{u}^{M}(0)\|_{1,M}+\|\bm{v}^{M}(0)\|_{1,M} \,  \leq \,  \C_0.$$
 Then, for any $T\geq 0$, there exist $c_1,c_2>0$ such that for any $M,N\geq 1$,
\begin{multline*}
  \EE\biggl( \sup_{t\in[0,T]} \bigl\|\bm{\mathcal U}^{M,N}(t) \bigr\|_{2,M}^2 + \sup_{t\in[0,T]}\bigl\|\bm{\mathcal V}^{M,N}(t) \bigr\|_{2,M}^2\biggr) \\ 
  \leq \biggl( \EE\Bigl(\bigl\|\bm{\mathcal U}^{M,N}(0)\bigr\|_{2,M}^2 + \bigl\|\bm{\mathcal V}^{M,N}(0) \bigr\|_{2,M}^2\Bigr) + c_1\biggl(\frac{M^2}{\sqrt{N}}+T\frac{M^3}{N}\biggr) \biggr) e^{c_2 M^4 T},
\end{multline*}
where $c_1$ and $c_2$ only depends on the diffusion parameters and the initial bounds.
\end{prop}
In particular, this estimate guarantees that the normalized stochastic process converges to the semi-discrete SKT system when the population size becomes large and the number of sites is fixed.
This constitutes an alternative approach for the rigorous derivation of the SKT system of \cite{chendausjunbis}, starting from discrete space. Both results seem to involve the same scales, with a number of individuals exponentially large compared to the inverse of the spatial scaling parameter. Our approach, in where the interaction is restricted to the same site, seems to  relax the condition of small cross-diffusion parameters in \cite{chendausjunbis}. Nevertheless, our main motivation in the rest of the paper is to go beyond this exponential scale and provide sharper estimates.

\begin{proof}[Proof of Proposition~\ref{prop:rateL2}] First,
using the fact that the total number of individuals
is constant along time, we  observe that  under our assumptions 
\begin{align}
\label{boundUV}
\max(\|\bm{U}^{M,N}(t)\|_{1,M},\|\bm{V}^{M,N}(t)\|_{1,M})&=\max(\|\bm{U}^{M,N}(0)\|_{1,M},\|\bm{V}^{M,N}(0)\|_{1,M} )  \leq  \C_0,
 \end{align}
 almost surely for any $M,N\geq 1$, and
\begin{align}
\label{bounduv}
\max(\|\bm{u}^{M}(t)\|_{1,M},\|\bm{v}^{M}(t)\|_{1,M})&=\max(\|\bm{u}^{M}(0)\|_{1,M},\|\bm{v}^{M}(0)\|_{1,M} )  \leq  \C_0,
 \end{align}
 for any $M\geq 1$.
Combining \eqref{eq:semimg_u} and \eqref{eq:SKTsemidis}, we notice that the process $\bm{\mathcal U}^{M,N}(t) = \bm{U}^{M,N}(t) - \bm{u}^M(t)$ has finite variations and satisfies
\begin{align*}
  \bm{ \mathcal U}^{M,N}(t)&= \bm{\mathcal U}^{M,N}(0)+
  \int_0^t d_1\Delta_M \bm{\mathcal U}^{M,N}(s) \,\dd s \\
  &\qquad  + \int_0^t a_{12}\Delta_M\bigl(\bm{U}^{M,N}(s)\odot \bm{V}^{M,N}(s)- \bm{u}^{M}(s)\odot\bm{v}^{M}(s) \bigr) \,\dd s +\bm{\mathcal M}^{M,N}(t).
\end{align*}
Consider now the square of its coordinates
  \begin{equation*}
    \mathcal U_i^{M,N}(t)^2 = \mathcal U_i^{M,N}(0)^2 + \int_0^t 2\,\mathcal U_i^{M,N}(s^-) \,\dd \mathcal U_i^{M,N}(s) +R_i^{M,N}(t),
  \end{equation*}
  for $i=1,\dots,M$, where 
  \begin{align*}
    R_i^{M,N}(t)&=  \sum_{0<s\leq t}\Bigl\lbrace \mathcal U_i^{M,N}(s) ^2 -  \mathcal U_i^{M,N}(s^-)^2  - 2\,\mathcal U_i^{M,N}(s^-) \bigl(\mathcal U_i^{M,N}(s) - \mathcal U_i^{M,N}(s^-)\bigr)\Bigl\rbrace.
  \end{align*}
  Putting the two last expressions together yields
  \begin{multline*}
    \mathcal U_i^{M,N}(t) ^2 = \mathcal U_i^{M,N}(0) ^2 + 2d_1\int_0^t \mathcal U_i^{M,N}(s) \bigl(\Delta_{M}\bm{\mathcal U}^{M,N}(s)\bigr)_i\,\dd s \\
    \hspace{3em}\qquad + 2a_{12}\int_0^t \mathcal{U}_i^{M,N}(s)\bigl(\Delta_{M}\bigl(\bm{U}^{M,N}(s)\odot \bm{V}^{M,N}(s)-\bm{u}^M(s)\odot \bm{v}^M(s)\bigr)\bigr)_i \,\dd s \\
    + 2\int_0^t \mathcal{U}_i^{M,N}(s^-) \,\dd \mathcal M_i^{M,N}(s) + R_i^{M,N}(t).
  \end{multline*}
  Given $\bm{u}\in\RR^M$ let us introduce the discrete gradient vector $\nabla_M^+\bm{u} = (M(u_{i+1}-u_{i}))_{1\leq i\leq M}$ (recalling the periodic convention).
  Summing over all the sites $i\in \{1, \ldots, M\}$ and
  using discrete integration by parts in the second and third terms of the right hand side
  yields 
  \begin{multline*}
    \bigl\|\bm{\mathcal U}^{M,N}(t) \bigr\|_2^2 = \bigl\|\bm{\mathcal U}^{M,N}(0) \bigr\|_2^2 - 2d_1\int_0^t \bigl\|\nabla_{M}^{+}\bm{\mathcal U}^{M,N}(s)\bigr\|_2^2\ \dd s \\
    - 2a_{12}\int_0^t \sum_{i=1}^M \bigl(\nabla_{M}^{+}\bm{\mathcal{U}}^{M,N}(s)\bigr)_i  \bigl(\nabla_{M}^{+}\bigl(\bm{U}^{M,N}(s)\odot \bm{V}^{M,N}(s)-\bm{u}^M(s)\odot \bm{v}^M(s)\bigr)\bigr)_i \,\dd s \\
    + 2\sum_{i=1}^M \int_0^t \mathcal U_i^{M,N}(s^-) \ \dd \mathcal M_i^{M,N}(s) +   \bigl\| \bm{R} ^{M,N}(t)   \bigl\|_1.
  \end{multline*}
  Dropping the second term which is negative, taking absolute value in the third term and using $2\vert ab\vert \leq \vert a\vert^2+ \vert b\vert^2$   ensures that 
  \begin{multline*}
    \bigl\|\bm{\mathcal U}^{M,N}(t) \bigr\|_2^2   \leq \bigl\|\bm{\mathcal U}^{M,N}(0) \bigr\|_2^2 + a_{12}\int_0^t \bigl\|\nabla_{M}^{+}\bm{\mathcal U}^{M,N}(s) \bigr\|_2^2 \,\dd s \\
    \hspace{2em}\qquad + a_{12}\int_0^t \bigl\|\nabla_{M}^{+}\bigl(\bm{U}^{M,N}(s)\odot \bm{V}^{M,N}(s) - \bm{u}^M(s)\odot \bm{v}^M(s)\bigr)\bigr\|_2^2 \,\dd s \\
    +  2\sum_{i=1}^M\int_0^t\mathcal U_i^{M,N}(s^-) \,\dd \mathcal M_i^{M,N}(s) + \bigl\| \bm{R} ^{M,N}(t)   \bigl\|_1.
  \end{multline*}
%
Moreover
  \begin{align*}
    R_i^{M,N}(t)&=  \sum_{0<s\leq t}\, \left(\mathcal U_i^{M,N}(s)  -  \mathcal U_i^{M,N}(s^-)\right)^2 =\Bigl(\frac{1}{N}\Bigr)^2 \sum_{0<s\leq t}\1_{ U_i^{M,N}(s)\ne U_i^{M,N}(s-)},
  \end{align*}
   since the jumps of $\mathcal U_i^{M,N}$ and $U_i^{M,N}$ coincide and are  of size $1/N$.
Then $\bigl\| \bm{R} ^{M,N}(t)   \bigl\|_1$ is given by the number of jumps before time $t$ 
 \[
   \mathbb E\left(\bigl\| \bm{R}^{M,N}(t)   \bigl\|_1\right)=2N^{-2}\mathbb E( \#\{t\geq 0 :  \bm{U}^{M,N}(s) \neq \bm{U}^{M,N}(s^-)\}).
 \]
 Moreover, the total jump rate in the scaled process $\bm{\mathcal U}^{M,N}$,
 when the number of individuals of each species in site $i$ is equal to $(u_i,v_i)$, is
 \[ 
   2M^2\sum_{i=1}^M u_i\Bigl(d_1+a_{12} \frac{v_i}{N}\Bigr)\leq  2M^2\| \bm {u}\|_1 \Bigl(d_1+a_{12} \frac{\| \bm{v}\|_1}{N}\Bigr)\leq
   C_0'M^3N(1+M),
 \]
 where $C_0'=2(d_1+a_{12})C_0$, by \eqref{boundUV}. Then we get
 \[
   \mathbb E\left(\bigl\| \bm{R} ^{M,N}(t)   \bigl\|_1 \right)\leq  2C_0'\, t\, \frac{M^3}{N}(1+M).
 \]
 Lets us now deal with the third and fourth terms. We notice that
  \begin{align*}
    \bigl(\nabla_{M}^{+} \bm{\mathcal{U}}^{M,N}(s)\bigr)_i^2 =& M^2\Bigl(\mathcal U_{i+1}^{M,N}(s)  - \mathcal U_i^{M,N}(s) \Bigr)^2 
    \leq  2M^2\left(\mathcal U_{i+1}^{M,N}(s)^2 + \mathcal U_i^{M,N}(s)^2\right),
  \end{align*}
Similarly, using also $\vert ab-cd\vert \leq \vert a-c\vert b +c\vert b-d\vert$ to deal with the difference of products of positive terms and recalling   \eqref{boundUV} and \eqref{bounduv},
we get 
  \begin{align*}
    \bigl( \nabla_{M}^{+}(\bm{U}^{M,N}(s) \odot \bm{V}^{M,N}(s) &- \bm{u}^{M}(s)\odot \bm{v}^M(s)) \bigr)_i^2 \\
    & \leq 4M^2 \Big(\|\bm{u}^M(0) \|_1^2 \,  \mathcal V_{i+1}^{M,N}(s)^2 + \|\bm{u}^M(0) \|_1^2 \, \mathcal V_i^{M,N}(s) ^2  \\
    & \qquad +\|\bm{V}^{M,N}(0)\|_1^2 \, \mathcal{U}_{i+1}^{M,N}(s)^2 
    + \|\bm{V}^{M,N}(0)\|_1^2 \, \mathcal U_i^{M,N}(s)^2 \Big)\\
    & \leq 4\C_0^2M^4 \Big(  \mathcal V_{i+1}^{M,N}(s)^2 + \, \mathcal V_i^{M,N}(s) ^2   + \, \mathcal{U}_{i+1}^{M,N}(s)^2 
    + \mathcal U_i^{M,N}(s)^2 \Big).
  \end{align*}
  Gathering these bounds, taking supremum and then expectation gives us
  \begin{align*}
   \EE\biggl(\sup_{s\in[0,t]} & \|\bm{\mathcal U}^{M,N}(s)\|_2^2\biggr) \\
   & \leq \EE\bigl(\|\bm{\mathcal U}^{M,N}(0) \|_2^2\bigr) + 4a_{12}M^2\int_0^t \EE\bigl( \|\bm{\mathcal U}^{M,N}(s) \|_2^2\bigr)\, \dd s \\
   & \qquad + 8\C_0^2a_{12}M^4 \left(\int_0^t\EE\bigl( \|\bm{\mathcal V}^{M,N}(s) \|_2^2\bigr) \,\dd s 
   +  \int_0^t\EE\bigl( \|\bm{\mathcal U}^{M,N}(s)\|_2^2\bigr)\, \dd s \right)\\
   &  \qquad\qquad + 2 \sum_{i=1}^M\EE\biggl(\sup_{s\in[0,t]}\int_0^s \mathcal U_i^{M,N}(r^-)\ \dd \mathcal M_i^{M,N}(r)\biggr) + 2C_0'\, T\, \frac{M^3}{N}(1+M),
  \end{align*}
    for some constant $C_0'$.
  For the martingale part, we use Cauchy-Schwarz and Burkholder-Davis-Gundy inequalities which together with \eqref{eq:predquad} and \eqref{boundUV} yield
  \begin{align*}
 & \EE\biggl(\sup_{s\in[0,t]}\int_0^s \mathcal U_i^{M,N}(r^-)\ \dd \mathcal M_i^{M,N}(r)\biggr)^2 \\
  & \qquad  \qquad \leq \EE\biggl(\sup_{s\in[0,t]} \biggl\vert \int_0^s \mathcal U_i^{M,N}(r^-) \, \dd\mathcal{M}_i^{M,N}(r) \biggr\vert^2\biggr)\\
    &  \qquad  \qquad \leq 
    \EE\biggl( \int_0^t \mathcal U_i^{M,N}(r^-)^2 \,\dd \bigl\langle \mathcal M_i^{M,N} \bigr\rangle(r)\biggr) \\
    & \qquad \qquad \leq 2\frac{M^2}{N} \EE\Bigl( \bigl\|\bm{U}^{M,N}(0)\bigr\|_1\Bigl(d_1+ a_{12}\ \bigl\|\bm{V}^{M,N}(0)\bigr\|_1\Bigr)  \int_0^t \mathcal U_i^{M,N}(s)^2 \,\dd s\Bigr)  \\
    &  \qquad \qquad \leq C_0'' \frac{M^3}{N} (1+M) \int_0^t \EE\Bigl(\mathcal U_i^{M,N}(s)^2\Bigr) \,\dd s,
  \end{align*}
for some constant $C_0''$.  Using that $\sqrt{x}\leq 1+x$ for all $x\geq0$, we obtain
  \begin{align*}
    \EE\biggl(\sup_{s\in[0,t]}\int_0^s \mathcal U_i^{M,N}(r^-) \, \dd \mathcal M_i^{M,N}(r)\biggr) \leq \sqrt{2C_0''} \frac{M^2}{\sqrt{N}}\Biggl( 1+\int_0^t \EE\Bigl(\mathcal U_i^{M,N}(s)^2\Bigr) \,\dd s \Biggr).
  \end{align*}
  Putting everything together and using again \eqref{boundUV} yields 
    \begin{equation*}
    \begin{split}
      \EE\biggl(\sup_{s\in[0,t]} \|\bm{\mathcal U}^{M,N}(s)\|_2^2\biggr)  &\leq  \EE\bigl(\|\bm{\mathcal U}^{M,N}(0) \|_2^2\bigr) + 2\sqrt{2C_0''}\frac{M^3}{\sqrt{N}} + 2C_0'T\frac{M^4}{N} \\
      &\quad + \biggl(8\C_0a_{12}M^4 + 2\sqrt{2C_0''}\frac{M^2}{\sqrt{N}}\biggr) \int_0^t\EE\biggl( \sup_{r\in[0,s]} \|\bm{\mathcal {U}}^{M,N}(r) \|_2^2\biggr)\, \dd s \\
      &\qquad\qquad\qquad\hspace{3em} + 8\C_0a_{12}M^4 \int_0^t \EE\biggl(\sup_{r\in[0,s]} \|\bm{\mathcal V}^{M,N}(r) \|_2^2\biggr)\, \dd s,
    \end{split}
  \end{equation*}
 In a similar way we can obtain analogous bounds for $\bm{V}^{M,N}$. Adding the two inequalities and then applying Gronwall's lemma leads us to the desired conclusion.
\end{proof}

 
   To go beyond the previous estimates, we will rely on a stability property for the SKT system that we will prove in the next section. This will allow us to compare the terms involved in the stochastic process to those of the targeted SKT system so that the former will appear as a stable perturbation of the latter.

\section{Duality estimates}\label{sec:duality}
\subsection{The continuous setting}
\label{subsec:contset}
The duality lemma is a tool first introduced by Martin, Pierre and Schmitt \cite{mapi,pisc}, in the context of reaction-diffusion systems. We propose below a small generalization of the duality lemma, which was suggested in \cite[Remark 7]{moussa_nonloc_tri}. As a matter of fact, we will not directly use the duality lemma presented in this paragraph, but rather translate it in a discrete setting (see Subsection~\ref{subsec:disdualem} below). 

\begin{lem}\label{lem:duamodbis}
  Consider $\mu\in L^\infty(Q_T)$ such that $\alpha\coloneqq\inf_{Q_T} \mu >0$, $z_0\in H^{-1}(\T^d)$ and $f\in L^2(Q_T)$. Then, there exists a unique $z\in L^2(Q_T)$ that solves weakly the Kolmogorov equation
\begin{equation}
  \label{eq:kol}          
  \left\{                                      
    \begin{lgathered}                                
    \partial_t z -\Delta(\mu z) = \Delta f,\\
    z(0,\cdot) =z_0.
    \end{lgathered}                                  
  \right.                                
\end{equation}
Furthermore, this solution $z$ belongs to $\mathscr{C}([0,T];H^{-1}(\T^d))$ and satisfies the \textsf{duality estimate} 
    \begin{align}\label{ineq:dua}
\|z(T)\|_{H^{-1}(\T^d)}^2 +\int_{Q_T} \mu z^2 \leq \|z_0\|_{H^{-1}(\T^d)}^2 + \meand{z_0}^2 \int_{Q_T} \mu + \frac{1}{\alpha}\int_{Q_T} f^2.
    \end{align}
  \end{lem}
  \begin{rmk}
  This duality estimate is stronger than the one stated in \cite{moussa_nonloc_tri}: it contains a (singular) source term and allows a uniform-in-time control of the $H^{-1}(\mathbb{T}^d)$ norm. The proof that we follow \emph{via} negative Sobolev energy estimate was used in \cite[Lemma 22]{perthame_laamri} in a different context, but only at the formal level (in a smooth setting). Here we include  a singular r.h.s. and give a well-posedness result in this rather non-smooth setting to justify all the computations.
    \end{rmk}
\begin{proof}
  The proof of existence and uniqueness is exactly the same as \cite[Theorem 3]{moussa_nonloc_tri}: following the naming of this article, $z$ is the unique dual solution of \eqref{eq:kol}. For this $z$, the regularity $\mathscr{C}([0,T];H^{-1}(\T^d))$ is obtained classically. We can thus focus here on the duality estimate which needs to be proven only in the case when every function involved in \eqref{ineq:dua} is smooth, in the sense that they are $\mathscr{C}^{\infty}$. For any $t\in [0,T]$ there exists a unique $\phi(t)$ of vanishing mean such that $-\Delta \phi(t) = z(t)- \meand{z(t)}$. Besides, by integrating the Kolmogorov equation we get
  \begin{align*}
\frac{\dd}{\dd t} \meand{z(t)} = 0,
  \end{align*}
so that $\meand{z(t)} = \meand{z_0}$ and $-\partial_t \Delta \phi = \partial_t z$. In particular, we have by integration by parts
\begin{align*}
\int_{\T^d} \phi(t)\, \partial_t z(t) = \frac12 \frac{\dd}{\dd t} \int_{\T^d} |\nabla \phi(t)|^2.
  \end{align*}
  Therefore, multiplying equation \eqref{eq:kol} by $\phi$ and using integration by parts 
  \begin{align*}
\frac12 \frac{\dd}{\dd t} \int_{\T^d} |\nabla \phi(t)|^2 + \int_{\T^d} \mu z (z-\meand{z_0}) =-\int_{\T^d} (z-\meand{z_0}) f.
  \end{align*}
    Integrating in time and using Young's inequality for the right hand side, we get
    \begin{equation*}
      \begin{split}
      \frac12 \int_{\T^d} |\nabla \phi(T)|^2 + \int_{Q_T} \mu z^2 &\leq  \int_{Q_T} \mu z\meand{z_0} + \frac12 \int_{\T^d} |\nabla \phi(0)|^2 \\ 
      &\qquad + \frac12  \int_{Q_T} (z-\meand{z_0})^2 \mu  + \frac12 \int_{Q_T}  \frac{f^2}{\mu},
      \end{split}
    \end{equation*}
    and thus, using $\mu\geq \alpha >0$,
    \begin{align*}
      \int_{\T^d} |\nabla \phi(T)|^2 + \int_{Q_T} \mu z^2 \leq \int_{\T^d} |\nabla \phi(0)|^2 + \meand{z_0}^2 \int_{Q_T} \mu  + \frac{1}{\alpha}\int_{Q_T}  f^2.
    \end{align*}
    Noticing that  $\|z(t)\|_{\dot{H}^{-1}(\T^d)}=\|z(t)-\meand{z_0}\|_{H^{-1}(\T^d)} = \|\nabla \phi(t)\|_2$, once we add $\meand{z_0}$ to each side of the inequality to get the full $H^{-1}(\T^d)$ norms, the proof is over.
\end{proof}
In Subsection~\ref{subsec:disdualem}, we will give (in the discrete setting) variants of the previous duality lemma which include in the r.h.s. some error term, which is possibly singular in the time variable. Being able to take into account those error terms will be crucial in the final asymptotic limit studied in Section~\ref{sec:quantification}. However, already in its current form, the previous duality lemma is a valuable piece of information. We highlight this with an application of this lemma: the proof of  Theorem~\ref{thm:sktstab}, which applies to the conservative SKT system \eqref{eq:SKT:cons} that we consider here with $(u_0,v_0)$ as initial data. We recall the definition of the affine functions $\mu_i(x)\coloneqq d_i+a_{ij}x$ for $i,j=1,2$, so that \eqref{eq:SKT:cons} rewrites
\begin{equation*}
  \left\{                                      
    \begin{lgathered}                                
      \partial_t u - \Delta( \mu_1(v)u ) = 0, \\
    \partial_t v - \Delta( \mu_2(u)v ) = 0.
    \end{lgathered}                                  
  \right.                                
\end{equation*}
In particular, we recover the framework of Lemma~\ref{lem:duamodbis}, as soon as $v$ and $u$ are bounded and non-negative. 
\begin{proof}[Proof of Theorem~\ref{thm:sktstab}]
Let's introduce $z\coloneqq\overline{u}-u$ and $w\coloneqq\overline{v}-v$, so that, by subtraction
  \begin{align*}
    \partial_t z -\Delta(\mu_1(\overline{v}) z) &= \Delta f,\\
    \partial_t w -\Delta(\mu_2(\overline{u})w) &= \Delta g,
    \end{align*}
    where $f\coloneqq a_{12}u(\overline{v}-v)$ and $g\coloneqq a_{21}v(\overline{u}-u)$. Since $u$ and $v$ are bounded and non-negative, we recover the structure of Lemma~\ref{lem:duamodbis} 
    and we get
    \begin{align*}
      \|z(T)\|_{H^{-1}(\T^d)}^2 + d_1\int_{Q_T} z^2 &\leq  \|z_0\|_{H^{-1}(\T^d)}^2 + \meand{z_0}^2 \int_{Q_T} \mu_1(\overline{v}) + \frac{a_{12}^2}{d_1}\|u\|_{ L^\infty(Q_T)}^2\int_{Q_T} w^2,\\
      \|w(T)\|_{H^{-1}(\T^d)}^2 + d_2\int_{Q_T} w^2 &\leq \|w_0\|_{H^{-1}(\T^d)}^2 + \meand{w_0}^2 \int_{Q_T} \mu_2(\overline{u}) + \frac{a_{21}^2}{d_2}\|v\|_{ L^\infty(Q_T)}^2\int_{Q_T} z^2,
    \end{align*}
    since $\inf_{Q_T} \mu_i\geq d_i$, $\vert f\vert \leq a_{12} \vert w\vert \,\|\overline{u}\|_{ L^\infty(Q_T)}$
    and $\vert g\vert \leq a_{21} \vert z\vert\, \|\overline{v}\|_{ L^\infty(Q_T)}$.
    By combining the two inequalities we infer
\begin{align*}
  \|z(T)\|_{H^{-1}(\T^d)}^2 + d_1 \int_{Q_T} z^2 &\leq \|z_0\|_{H^{-1}(\T^d)}^2 + \meand{z_0}^2 \int_{Q_T} \mu_1(\overline{v})\\
  &\qquad + \frac{a_{12}^2}{d_1 d_2}\|u\|_{ L^\infty(Q_T)}^2\Big( \|w_0\|_{H^{-1}(\T^d)}^2 
  +  \meand{w_0}^2 \int_{Q_T}\mu_2(\overline{u}) \Big)\\
  &\qquad\qquad\qquad + d_1\left(\frac{a_{12}a_{21}}{d_1 d_2}\right)^2 \|u\|_{ L^\infty(Q_T)}^2\|v\|_{L^\infty(Q_T)}^2 \int_{Q_T} z^2.
\end{align*}
In particular, if we want to absorb the last term of the r.h.s. in the l.h.s. the inequality that we need is exactly the smallness condition \eqref{ineq:small}. If the later is satisfied, and if we allow the symbol $\lesssim$ to depend on $d_i,a_{ij},\|u\|_{L^\infty(Q_T)}$ and $\|v\|_{L^\infty(Q_T)}$, we have  established
\begin{equation*}
  \begin{split}
    \|z(T)\|_{H^{-1}(\T^d)}^2 + \int_{Q_T} z^2 &\lesssim \|z_0\|_{H^{-1}(\T^d)}^2 + \|w_0\|_{H^{-1}(\T^d)}^2 \\
    & \qquad + \meand{z_0}^2 \int_{Q_T} \mu_1(\overline{v}) +  \meand{w_0}^2 \int_{Q_T}\mu_2(\overline{u}).
  \end{split}
\end{equation*}
  Since the previous computation is still valid replacing $T$ by any $t\in[0,T]$, we have in fact
\begin{align*}
  \vertiii{z}_T^2 \lesssim \|z_0\|_{H^{-1}(\T^d)}^2 + \|w_0\|_{H^{-1}(\T^d)}^2+\meand{z_0}^2 \int_{Q_T} \mu_1(\overline{v}) +  \meand{w_0}^2 \int_{Q_T}\mu_2(\overline{u}).
  \end{align*}
  Exchanging the roles $(z,\overline{u},\overline{v},u,v) \leftrightarrow (w,\overline{v},\overline{u},v,u)$, the previous right hand side remains unchanged: we have exactly the same estimate for $\vertiii{w}^2_T$ on the left hand side. The proof is over once we notice that
$\int_{Q_T}\mu_1(\overline{v})=T \int_{\T^d}\mu_1(\overline{v}_0)$ and $\int_{Q_T}\mu_2(\overline{u}) =T\int_{\T^d}\mu_2(\overline{u}_0)$,
since the space integrals of $u$ and $v$ are conserved through time.  
\end{proof}

\subsection{Reconstruction operators}
We now transfer the previous estimates into a discrete setting. We will have to manipulate several norms on $\RR^M$, reminiscent of classical function spaces of the continuous variable. As the number of points $M$ of the discretization will be sent to infinity, it will be crucial to have estimates which do not depend on this parameter. In particular, the following notion of uniform equivalence will be relevant.
\begin{dfn}
  Given norms $P_{1,M}$ and $P_{2,M}$ on $\RR^M$, we say that $P_{1,M}$ and $P_{2,M}$ are \textsf{uniformly equivalent} if there exists $\alpha,\beta>0$ such that
  \[\forall M\in\NN,\quad\forall \bm{u}\in\RR^M,\quad \alpha P_{1,M}(\bm{u})\leq P_{2,M}(\bm{u})\leq \beta P_{1,M}(\bm{u}).\]
  If this is satisfied, we write $P_{1,M}\sim P_{2,M}$.
\end{dfn}
 Given a discretization like \eqref{eq:discT}, we will use two interpolation methods to build a function defined on the whole torus $\T$.
\begin{dfn}
For $\bm{u}\in\RR^M$,
the function defined for $x\in \T$ by
$$\sigma_M(\bm{u})(x):= \sum_{k=1}^M \mathbf{1}_{[-1,0]}\left(M(x-x_k)\right)\,u_k,$$
 is a step function and
 the function
 $$\pi_M(\bm{u})(x):=\sum_{k=1}^M \theta\left(M(x-x_k)\right)\,u_k \,, 
 \qquad \text{where } \, \theta(z)\coloneqq(1-|z|)^+,$$
 is a piecewise linear function. The corresponding vector space of functions (step and continuous piecewise linear functions respectively) are denoted
\begin{align*}
  \mathfrak{s}_M &\coloneqq \left\{\sigma_M(\bm{u}) : \bm{u}\in\RR^M\right\}\quad\text{and}\quad
  \mathfrak{p}_M \coloneqq \left\{\pi_M(\bm{u}) : \bm{u}\in\RR^M\right\}.
\end{align*}
If $t\mapsto \bm{u}(t)$ is a map from $[0,T]$ to $\RR^M$, we simply denote by $\sigma_M(\bm{u})$ and $\pi_M(\bm{u})$ the respective maps from $[0,T]$ to $\mathfrak{s}_M$ and $\mathfrak{p}_M$ respectively. 
\end{dfn}

\begin{prop}\label{prop:normlp}
For $\bm{u}\in\RR^M$ we have $\|\bm{u}\|_\infty = \|\sigma_M(\bm{u})\|_{L^\infty(\T)}=\|\pi_M(\bm{u})\|_{L^\infty(\T)}$ and for $1\leq p<\infty$ we have $\|\bm{u}\|_{p,M} = \|\sigma_M(\bm{u})\|_{L^p(\T)} \geq 
\|\pi_M(\bm{u})\|_{L^p(\T)}$. Furthermore, the equivalence $\|\sigma_M(\cdot)\|_{L^p(\T)}\sim \|\pi_M(\cdot)\|_{L^p(\T)}$ holds
on the positive cone $\RR_+^M$. 
\end{prop}

\begin{proof}
  The equalities are obvious. For the inequality and the uniform equivalence, we refer to \cite[Lemma 11]{DDD2019}.
      
\end{proof}

We end this paragraph with an estimate that belongs to the folklore of the finite element method and omit the proof. It is usually proved using the Bramble-Hilbert lemma, but since here we focus on the one dimensional case, it is also possible to give a direct, elementary proof (see for instance \cite[Lemma 6.2.10]{allaire}).
\begin{lem}\label{lem:bh}
  For $\ffi\in H^2(\T)$ and $M\in\NN^*$ there exists a unique $\iota_M(\ffi)\in\mathfrak{p}_M$ matching the values of $\ffi$ on the grid $(x_k)_{1\leq k\leq M}$. It satisfies
  \begin{align*}
 \|\ffi-\iota_M(\ffi)\|_{\dot{H}^{-1}(\T)} &\lesssim M^{-2}\|\ffi\|_{\dot{H}^2(\T)},\\  
    \|\ffi-\iota_M(\ffi)\|_{L^2(\T)} &\lesssim M^{-2}\|\ffi\|_{\dot{H}^2(\T)},\\
    \|\ffi-\iota_M(\ffi)\|_{\dot{H}^1(\T)} &\lesssim M^{-1}\|\ffi\|_{\dot{H}^2(\T)},
  \end{align*}
  where the symbol $\lesssim$ means that the inequality holds up to a constant independent of $\ffi$ and $M$.
\end{lem}

\subsection{A discrete negative Sobolev norm} \label{subsec:A}
We introduce in this paragraph a norm on $\RR^M$ analogous to the $H^{-1}(\T^d)$ norm for functions.  
We summarize first the main (standard) properties of the laplacian matrix  $\Delta_M$ introduced in \eqref{eq:DeltaM} in the following proposition (for a proof see for instance \cite{smith}) .
\begin{prop}\label{prop:deltamat}
  Recalling the definition \eqref{eq:DeltaM}, the spectrum of the matrix $-\Delta_M$ is given by 
  \begin{align*}
    \left\{4M^2\sin^2\hspace{-0.2em}\left(\frac{\pi k}{M}\right): 0\leq k \leq M-1\right\} \subset \RR_+.
  \end{align*}
  We have thus $-\Delta_M\in\textnormal{S}_M^+(\RR)$ and this matrix admits therefore a unique symmetric non-negative square root. One has furthermore $\textnormal{Ker}(\Delta_M)=\textnormal{Span}_\RR(\mathbf{1}_M)$ and $-\Delta_M$ enjoys a uniform (in $M$) spectral gap : all non-zero eigenvalues of $-\Delta_M$ are lower-bounded by $16$, independently of the dimension $M$.
\end{prop}
Using this, we fix the following notations. 
\begin{dfn}\label{def:A}
  For $\bm{u}\in \mathrm{Ran}(\Delta_M)$ the unique $\Phi\in \mathrm{Ran}(\Delta_M)$ such that $\bm{u}=\Delta_M\Phi$ is denoted (with a small abuse of notation) $\Phi=\Delta_M^{-1}\bm{u}$. The square root of the non-negative matrix $-\Delta_M$ is denoted $\sqrt{-\Delta_M}$. 
\end{dfn}
The uniform spectral gap for the discrete laplacian (see Proposition~\ref{prop:deltamat}) implies in particular the following estimate for any $\bm{\Phi}\in\RR^M$ 
\begin{align}\label{prop:pw}
\|\bm{\Phi}- \meanM{\bm{\Phi}}\|_{2,M} \leq \|\Delta_M\bm{\Phi}\|_{2,M}.
\end{align}
On the torus, the Poincaré-Wirtinger inequality implies the estimate $\|\ffi-\mean{\ffi}\|_{L^2(\T)} \lesssim \|\Delta \ffi\|_{L^2(\T)}$, from which the previous inequality is somehow reminiscent.

\medskip

A standard computation when dealing with the Lagrange finite element methods in dimension $1$ shows that, up to a factor $1/M$, the stiffness matrix is precisely given by $-\Delta_M$ whereas the mass matrix is given by (see Section 6.2.1 and Exercise 7.4.1 in \cite{allaire})
\begin{equation}\label{eq:BN}
  B_M\coloneqq
  \begin{pmatrix}
    \frac23 & \frac16 & 0 &\cdots& \frac16 \\
    \frac16 & \frac23 & \frac16 &\cdots& 0\\
    \vdots & \ddots &\ddots & \ddots &\vdots \\
    0 & \cdots & \frac16 & \frac23 & \frac16\\
    \frac16 & \cdots & 0 & \frac16 & \frac23
  \end{pmatrix}.
\end{equation}
More precisely, recalling that $\ffi_{k,M}(x)\coloneqq\ffi(M(x-x_k))$ where $\ffi(x) \coloneqq (1-|x|)^+$,  we have \begin{align*}
(-\Delta_M)_{k,j} = M \int_{\T} \nabla\ffi_{k,M} \cdot \nabla \ffi_{j,M},\qquad (B_M)_{k,j}=\frac{1}{M}\int_{\T} \ffi_{k,M}\ffi_{j,M},
\end{align*}
for any $1\leq k,j\leq M$.
Since $\mathfrak{p}_M$ is the vector space spanned by the functions $(\ffi_{k,M})_{1\leq k\leq M}$, expanding elements of this space on that basis we recover the following standard result.
\begin{prop}\label{prop:elemfin}
  For $\bm{w}\in\RR^M$ we have
  \begin{equation}\label{eq:elemfin2}
    -(\bm{w}|\Delta_M\bm{w})_M = \int_\T |\nabla \pi_M(\bm{w})(x)|^2 \,\dd 
x,
  \end{equation}
  where we recall that $(\cdot|\cdot)_M$ denotes the rescaled inner product on $\RR^M$ (see Subsection~\ref{subsec:not}). Furthermore, for any $\bm{u}\in\RR^M$ we have
  \begin{equation}\label{eq:elemfin1}
    B_M \bm{u} = - \Delta_M \bm{w}
     \, \Longleftrightarrow\,
    \forall \psi\in\mathfrak{p}_M, \int_\T \psi(x) \, \pi_M(\bm{u})(x) \,\dd x = \int_\T \nabla \psi(x) \cdot \nabla \pi_M(\bm{w})(x) \,\dd x.
  \end{equation}
\end{prop}

Recalling $\widetilde{\bm{u}}=\bm{u}-\meanM{\bm{u}}\mathbf{1}_M$ and Definition~\ref{def:A}, we infer from Proposition~\ref{prop:deltamat} that $-(\widetilde{\bm{u}}|\Delta_M^{-1}\widetilde{\bm{u}})_M\geq 0$. This enables us to introduce the following norm $\|\cdot\|_{-1,M}$, which is a discrete counterpart of the $H^{-1}(\T)$ norm.
\begin{dfn}\label{def:h-1} For $\bm{u}\in\RR^M$, we
define 
  \begin{align*}
    \|\bm{u}\|_{-1,M}\coloneqq\sqrt{-(\widetilde{\bm{u}}|\Delta_M^{-1}\widetilde{\bm{u}})_M + \meanM{\bm{u}}^2}.
  \end{align*}
This is a hilbertian norm on $\RR^M$, whose associated inner-product is given by the following formula, for $\bm{u},\bm{v}\in\RR^M$: 
\begin{align*}
    (\bm{u}|\bm{v})_{-1,M} := (\widetilde{\bm{u}}|\Delta_M^{-1}\widetilde{\bm{v}})_M +\meanM{\bm{u}}\meanM{\bm{v}}.
\end{align*}
\end{dfn}
\begin{prop}\label{prop:defsobneg}       
We have the uniform equivalence:
  \begin{align}\label{sim:equiv}      
  M  \|\pi_M(\cdot)\|_{H^{-1}(\T)}+ \|\pi_M(\cdot)\|_{L^2(\T)} \sim M \|\cdot\|_{-1,M} +\|\pi_M(\cdot)\|_{L^2(\T)}.
  \end{align}
Moreover 
for any $\bm{u}\in\RR^M$, 
  \begin{align}\label{borne-12}
 \|\bm{u}\|_{-1,M} \, \leq  \,   \|\bm{u}\|_{2,M}.
  \end{align} 
\end{prop}

\begin{proof}
          We first observe the uniform equivalences
  \begin{align*}
    \|\pi_M(\bm{u})\|_{L^2(\T)} &\sim \|\pi_M(\widetilde{\bm{u}})\|_{L^2(\T)}+|\meanM{\bm{u}}|,\\
    \|\pi_M(\bm{u})\|_{H^{-1}(\T)} &\sim \|\pi_M(\widetilde{\bm{u}})\|_{H^{-1}(\T)}+|\meanM{\bm{u}}|,\\
    \|\bm{u}\|_{-1,M} &\sim \|\widetilde{\bm{u}}\|_{-1,M} + |\meanM{\bm{u}}|.
  \end{align*}
  Without loss of generality we can therefore establish the uniform equivalence \eqref{sim:equiv} under the assumption $\meanM{\bm{u}} = 0$.

  We have $\|\bm{u}\|_{-1,M}^2 = - (\bm{u}|\Delta_M^{-1}\bm{u})_M=-(\Delta_M\bm{\Phi},\bm{\Phi})_M$ where $\bm{\Phi}\coloneqq-\Delta_M^{-1}\bm{u}$. Thanks to Proposition~\ref{prop:elemfin} we have therefore
\begin{align}\label{eq:hnegM}
  \|\bm{u}\|_{-1,M}^2 = \|\nabla \pi_M(\bm{\Phi})\|_{L^2(\T)}^2.
\end{align}
The matrix $B_M$ defined by \eqref{eq:BN} satisfies $6B_M = M^{-2}\Delta_M+6\textnormal{I}_M$, so it commutes  with $\Delta_M$. In particular, the equation $\bm{u}=-\Delta_M\bm{\Phi}$ is strictly equivalent to 
  \begin{equation*}
    B_M \bm{u} = -  \Delta_M \bm{w},
  \end{equation*}
  where $\bm{w}\coloneqq B_M\bm{\Phi}$. We obtain from Proposition~\ref{prop:elemfin} that this last equation is equivalent to
  \begin{equation*}
    \forall \psi\in \mathfrak{p}_M,\quad \int_{\T} \psi(x)\,\pi_M(\bm{u})(x) \,\dd x = \int_{\T} \nabla\psi(x) \cdot \nabla \pi_M(\bm{w})(x) \,\dd x.
  \end{equation*}
  Since we assumed $\meanM{\bm{u}} = 0$, we have also $\mean{\pi_M(\bm{u})}=0$ and we can therefore solve $-\Delta \ffi_M = \pi_M(\bm{u})$, for a unique $\ffi_M\in\dot{H}^2(\T)$. We have then, by integration by parts,
  \begin{equation*}
    \forall \psi\in \mathfrak{p}_M,\quad \int_{\T} \psi(x)\,\pi_M(\bm{u})(x) \,\dd x = \int_{\T}\nabla\psi(x) \cdot  \nabla \ffi_M(x) \,\dd x.
  \end{equation*}
  In particular, we have established
  \begin{equation*}
    \forall \psi \in \mathfrak{p}_M,\quad\int_{\T} \nabla \psi(x) \cdot (\nabla \pi_M(\bm{w})(x) - \nabla \ffi_M(x)) \,\dd x =0,
  \end{equation*}
  and this equality holds in particular for $\psi=\pi_M(\bm{w})$. We deduce that for each $\psi \in \mathfrak{p}_M$
  \begin{align*}
   & \int_{\T} |\nabla \pi_M(\bm{w})(x) - \nabla \ffi_M(x)|^2 \,\dd x\\
     &\qquad = \int_{\T} (\nabla\pi_M(\bm{w})(x)-\nabla \ffi_M(x)+\nabla\psi(x)-\nabla\pi_M(\bm{w})(x))\cdot (\nabla \pi_M(\bm{w})(x)-\nabla \ffi_M(x)) \,\dd x \\
    &\qquad = \int_{\T} (\nabla\psi(x)-\nabla \ffi_M(x))\cdot (\nabla \pi_M(\bm{w})(x)-\nabla \ffi_M(x)) \,\dd x,
  \end{align*}
  and we get by the Cauchy-Schwarz inequality 
  \[ \|\nabla\pi_M(\bm{w})-\nabla \ffi_M\|_{L^2(\T)} \leq \inf_{\psi\in\mathfrak{p}_M} \|\nabla \psi-\nabla \ffi_M\|_{L^2(\T)}. \]
  Taking $\psi=\iota_M(\ffi)$ and using successively 
  $\|\nabla f\|_{L^2(\T)} = 2\pi\|f\|_{\dot{H}^1(\T)}$ for  $f=\iota_M(\ffi)- \ffi_M\in \dot{H}^1(\T)$
  and  the third estimate of Lemma~\ref{lem:bh}, we get  
\begin{align*}
  \|\nabla\pi_M(\bm{w})-\nabla \ffi_M\|_{L^2(\T)} &\lesssim \|\nabla \iota_M(\ffi)-\nabla \ffi_M\|_{L^2(\T)} \lesssim \|\iota_M(\ffi)- \ffi_M\|_{\dot{H}^1(\T)} \lesssim \frac{1}{M}\|\ffi_M\|_{\dot{H}^2(\T)}.
\end{align*}
Recalling that  $-\Delta\ffi_M = \pi_M(\bm{u})$ we have  $\|\pi_M(\bm{u})\|_{\dot{H}^{-1}(\T)} = \|\nabla \ffi_M\|_{L^2(\T)}$ and $\|\ffi_M\|_{\dot{H}^2(\T)} = \|\Delta\ffi_M \|_{L^2(\T)}=  \|\pi_M(\bm{u})\|_{L^2(\T)}$.
All in all, using 
the reversed triangular inequality we have established
  \begin{align*}
    \Big|\|\nabla\pi_M(\bm{w})\|_{L^2(\T)}-\|\pi_M(\bm{u})\|_{\dot{H}^{-1}(\T)}\Big|\lesssim \frac{1}{M}\|\pi_M(\bm{u})\|_{L^2(\T)}.
  \end{align*}
  To conclude, due to \eqref{eq:hnegM}, it is thus sufficient to prove that $\|\nabla \pi_M(\bm{w})\|_{L^2(\T)} \sim \|\nabla \pi_M(\Phi)\|_{L^2(\T)}$, where we recall $\bm{w}=B_M \Phi$. This last equality implies in particular 
  \[
    \pi_M(\bm{w}) = \frac23 \pi_M(\Phi) + \frac16 \tau_{{\frac{1}{M}}}\pi_M(\Phi)+\frac16\tau_{-\frac{1}{M}}\pi_M(\Phi),
  \]
where we recall  the translation operator $\tau_a$ defined by $\tau_a f(x) = f(x+a)$. We have therefore
  \begin{align}\label{id:grad}
    \nabla \pi_M(\bm{w}) = \frac23 \nabla\pi_M(\Phi) + \frac16 \tau_{\frac{1}{M}}\nabla \pi_M(\Phi)+\frac16\tau_{-\frac{1}{M}}\nabla \pi_M(\Phi).
  \end{align}
Both $\nabla\pi_M(\bm{w})$ and $\nabla \pi_M(\Phi)$ belong to $\mathfrak{s}_M(\T)$ \emph{i.e.} are respectively equal to some functions $\sigma_M(\bm{\lambda})$ and $\sigma_M(\bm{\gamma})$, for some $\bm{\lambda},\bm{\gamma}\in\RR^M$. 

A classical computation shows (see \cite[Exercise 7.4.1]{allaire} for instance) that the spectrum of $B_M$ lies within $[1/3,1]$. In particular, the spectral radius of both $B_M$ and $B_M^{-1}$ are bounded independently of $M$. The identity \eqref{id:grad} shows that $\bm{\lambda} = B_M \bm{\gamma}$ and we have just controlled the euclidean subordinate norms of $B_M$ and $B_M^{-1}$: we have $\|\bm{\gamma}\|_{2,M} \sim \|B_M\bm{\gamma}\|_{2,M}$, and therefore $\|\nabla \pi_M(\bm{w})\|_{L^2(\T)} \sim \|\nabla \pi_M(\Phi)\|_{L^2(\T)}$, thanks to Proposition~\ref{prop:normlp}, concluding the proof of \eqref{sim:equiv}.
  
  Let us turn to the proof of \eqref{borne-12}.
  Using \eqref{prop:pw}, $\| \Delta_M^{-1}\widetilde{\bm{u}}\|_{2,M}\leq
  \| \widetilde{\bm{u}}\|_{2,M}$ and Cauchy-Schwarz inequality entails that
  $-(\widetilde{\bm{u}}|\Delta_M^{-1}\widetilde{\bm{u}})_M\leq
  \| \widetilde{\bm{u}}\|_{2,M}^2$. By  Pythagore's identity, we obtain $\eqref{borne-12}$, since $\bm{u}=\widetilde{\bm{u}}+\meanM{\bm{u}}
  \mathbf{1}_M$ and $\|\meanM{\bm{u}}\mathbf{1}_M\|_{2,M}^2=\meanM{\bm{u}}^2$.
\end{proof}

\begin{prop}\label{prop:calnorm}
  For $\bm{w}\in \mathscr{C}^1([0,T];\mathrm{Ran}(\Delta_M))$, we have
  \[-(\Delta_M^{-1}\bm{w}(t)|\bm{w}'(t))_M = \frac12\frac{\dd}{\dd t}\|\bm{w}(t)\|_{-1,M}^2.\]
\end{prop}

\begin{proof}
  If $\bm{v}(t)\coloneqq-\Delta_M^{-1}\bm{w}(t)$, we have $\Delta_M \bm{v}(t) = -\bm{w}(t)$ and therefore $\Delta_M \bm{v}'(t) = -\bm{w}'(t)$, with still $\meanM{\bm{v}'(t)} = 0$. We then have $\bm{v}'(t) = -\Delta_M^{-1}\bm{w}'(t)$. We infer, by symmetry of $\sqrt{-\Delta_M}$, 
  \begin{align*}
    - (\Delta_M^{-1}\bm{w}(t)|\bm{w}'(t))_M &= - \big(\bm{v}(t)|\Delta_M \bm{v}'(t)\big)_M\\
    &=  \Big(\sqrt{-\Delta_M}\bm{v}(t)|\sqrt{-\Delta_M}\bm{v}'(t)\Big)_M\\
                                                       &=\frac12\frac{\dd}{\dd t}  \Big(\sqrt{-\Delta_M}\bm{v}(t)|\sqrt{-\Delta_M}\bm{v}(t)\Big)_M\\
    &=-\frac12\frac{\dd}{\dd t}  (\bm{v}(t)|\Delta_M\bm{v}(t))_M  = \frac12\frac{\dd}{\dd t}  \|\bm{w}(t)\|_{-1,M}^2.
    \mbox{\qedhere}
  \end{align*}
\end{proof}

\subsection{The discrete duality lemma}\label{subsec:disdualem}
We are now all set to state and prove two discrete duality lemmas. 
They are counterparts of Lemma~\ref{lem:duamodbis}  in a semi-discrete setting and 
they are to be applied to an ODE. At the same time, they  generalize Lemma~\ref{lem:duamodbis} since they include an additionnal source term. We first consider the case when this source term is regular (Lemma~\ref{lem:duadis}) and then the case when it isn't (Lemma~\ref{lem:duasing}). The second case amounts to consider an ODE with several Dirac masses in the right hand side. Being able to handle this singular setting will be of crucial importance in order to use these results for stochastic jump processes (see Proposition~\ref{prop:core} below).

\begin{lem}\label{lem:duadis}
  Consider $\bm{\mu}\in \mathscr{C}([0,T];\RR_{>0}^M)$ so that each component is uniformly (w.r.t. to time and index) lower bounded by a positive constant $\alpha>0$. Consider $\bm{f},\bm{r}\in\mathscr{C}([0,T];\RR^M)$. There exists a unique function $z\in\mathscr{C}^1([0,T];\RR^M)$ solving, for some fixed $\bm{z}_0\in\RR^M$, 
  \begin{align*}
    \bm{z}'(t) &= \Delta_M\Big[\bm{z}(t)\odot \bm{\mu}(t)+\bm{f}(t)\Big]+\bm{r}(t),\\
    \bm{z}(0) &= \bm{z}_0.
  \end{align*}
This function satisfies furthermore 
  \begin{multline}\label{ineq:duadis}
    \sup_{t\in[0,T]}\|\bm{z}(t)\|_{-1,M}^2 +  \int_{Q_T} \sigma_M(\bm{z} \odot \bm{\mu}^{1/2})(s,x)^2 \,\dd s\,\dd x  \\
    \leq  \|\bm{z}_0\|_{-1,M}^2 + \int_{0}^T \meanM{\bm{z}(s)}^2 \meanM{\bm{\mu}(s)} \,\dd s\\ + \frac{1}{\alpha} \int_{Q_T} \sigma_M(\bm{f})(s,x)^2 \,\dd s\,\dd x 
    + 2\int_0^T ( \bm{z}(s)|\bm{r}(s))_{-1,M}\,\dd s,
  \end{multline}
where the Hadamard product $\odot$ and the square-root  $\bm{\mu}^{1/2}$ are defined in Subsection~\ref{subsec:not}.
\end{lem}
\begin{proof}
  Existence and uniqueness of $\bm{z}$ are straightforward, the ODE being linear with continuous coefficients. To get the estimate we first notice 
  \begin{align}\label{eq:dermean}
    \meanM{\bm{z}(t)}' = \meanM{\bm{r}(t)},
  \end{align}
  and therefore, recalling the notation $\widetilde{\bm{z}}(t):=\bm{z}(t)-\meanM{\bm{z}(t)}$,
  \begin{align*}
 \bm{z}'(t)=\widetilde{\bm{z}}'(t) +\meanM{\bm{r}(t)}.
  \end{align*}
Now, taking the inner-product of the ODE with the vector $\Delta_M^{-1}\widetilde{\bm{z}}(t)$, we get, using the symmetry of $\Delta_M$ and the fact $\Delta_M^{-1} \widetilde{\bm{z}}(t) \in \textnormal{Span}_\RR(\mathbf{1}_M)^\perp$ (see Subsection~\ref{subsec:A}), 
  \begin{equation*}
    -\Big(\Delta_M^{-1}\widetilde{\bm{z}}(t)\big|\widetilde{\bm{z}}'(t)\Big)_M + \Big(\widetilde{\bm{z}}(t)\big|\bm{z}(t)\odot \bm{\mu}(t)\Big)_M = -\Big(\widetilde{\bm{z}}(t)\big|\bm{f}(t)\Big)_M-\Big(\widetilde{\bm{z}}(t)\big|\Delta_M^{-1}\bm{\widetilde{r}}(t)\Big)_M.
  \end{equation*}
  We use Proposition~\ref{prop:calnorm} to identify the first term of the l.h.s. and get
\begin{equation}\label{eq:calnorm}
    \frac12 \frac{\dd}{\dd t} \|\widetilde{\bm{z}}(t)\|_{-1,M}^2 + \Big(\widetilde{\bm{z}}(t)\big|\bm{z}(t)\odot \bm{\mu}(t)\Big)_M =  -\Big(\widetilde{\bm{z}}(t)\big|\bm{f}(t)\Big)_M-\Big(\widetilde{\bm{z}}(t)\big|\Delta_M^{-1}\bm{\widetilde{r}}(t)\Big)_M.
  \end{equation}
  Using Cauchy-Schwarz's inequality and that the entries of $\bm{\mu}(t)$ are all lower-bounded by $\alpha>0$ we have the following inequality, for any vector $\bm{g}\in\RR^M$ (using the inequality $2ab\leq a^2+b^2$)
  \begin{align*}
    \left|\Big( \bm{\widetilde{z}}(t)\big| \bm{g}\Big)_M\right| =  \|\bm{\widetilde{z}}(t)\|_{2,M} \|\bm{g}\|_{2,M} \leq \frac12 \Big(\bm{\widetilde{z}}(t)|\bm{\widetilde{z}}(t)\odot \bm{\mu}(t)\Big)_M + \frac{1}{2\alpha}\|\bm{g}\|_{2,M}^2.
  \end{align*}

  Using this estimate in \eqref{eq:calnorm} with $\bm{g} \coloneqq \bm{f}(t)$ and the definition $\bm{\widetilde{z}}(t) \coloneqq \bm{z}(t)-\meanM{\bm{z}(t)}$ we get
  \begin{multline*}
    \frac12 \frac{\dd}{\dd t} \|\widetilde{\bm{z}}(t)\|_{-1,M}^2 +  \Big(\bm{z}(t)\Big|\bm{z}(t)\odot \bm{\mu}(t)\Big)_M \leq \meanM{\bm{z}(t)} \meanM{\bm{z}(t)\odot\bm{\mu}(t)} + \frac12 \Big(\bm{\widetilde{z}}(t)\Big|\bm{\widetilde{z}}(t)\odot \bm{\mu}(t)\Big)_M \\+ \frac{1}{2\alpha}\|\bm{f}(t)\|_{2,M}^2-\Big(\widetilde{\bm{z}}(t)\big|\Delta_M^{-1}\bm{\widetilde{r}}(t)\Big)_M.
  \end{multline*}
  Using once more the definition $\bm{\widetilde{z}}(t) \coloneqq \bm{z}(t)-\meanM{\bm{z}(t)}$ we get eventually 
  \begin{align*}
     \frac{\dd}{\dd t} \|\widetilde{\bm{z}}(t)\|_{-1,M}^2 +  \Big(\bm{z}(t)\Big|\bm{z}(t)\odot \bm{\mu}(t)\Big)_M \leq  \frac{1}{\alpha}\|\bm{f}(t)\|_{2,M}^2+\meanM{\bm{z}(t)}^2\meanM{\bm{\mu}(t)}-2\Big(\widetilde{\bm{z}}(t)\big|\Delta_M^{-1}\bm{\widetilde{r}}(t)\Big)_M.
  \end{align*}
  Now, note on one hand that $\widetilde{\bm{z}}(t) \perp \mathbf{1}_M$, for the $(\cdot|\cdot)_{-1,M}$ inner-product and on the other hand, because of \eqref{eq:dermean},
  \begin{align*}
\frac12 \frac{\dd}{\dd t} \meanM{\bm{z}(t)}^2 = \meanM{\bm{z}(t)}\meanM{\bm{r}(t)}. 
    \end{align*}
Adding this last quantity to both sides of the estimate and integrating in time, we recover \eqref{ineq:duadis}, since for any vector $\bm{u}\in\RR^M$,  $\|\bm{u}\|_{2,M} = \|\sigma_M(\bm{u})\|_{L^2(\T)}$. $\qedhere$
\end{proof}


For the next lemma, we introduce
$0<t_0<t_1<\cdots <t_m<T$ some (fixed) jump times and the intervals $I_k:=]t_k,t_{k+1}[$. We fix also a function $\bm{x}:[0,T]\rightarrow\RR^M$ which is c\`adl\`ag (continuous right and limited left) and $\mathscr{C}^1$ inside each of the intervals $I_k$ and has jump discontinuities  $\bm{a}_k:=\bm{x}(t_k)-\bm{x}(t_k^-)$ at each $t_k$. Note that such a function $\bm{x}$ can be decomposed as $\bm{x} = \bm{x}_R + \bm{x}_J$ where $\bm{x}_R\in\mathscr{C}^1([0,T];\RR^M)$ is the regular part and $\bm{x}_J$ is the jump part, explicitely given by \[\bm{x}_J(t) := \sum_{k=1}^m \bm{a}_k \mathbf{1}_{t\geq t_k}.\]

\begin{lem}\label{lem:duasing}
  Consider $\bm{\mu}\in \mathscr{C}([0,T];\RR_{>0}^M)$ so that each component is uniformly (w.r.t. to time and index) lower bounded by a positive constant $\alpha>0$. Consider $\bm{f}\in\mathscr{C}([0,T];\RR^M)$ and $\bm{x}$ a c\`adl\`ag piecewise $\mathscr{C}^1$ function just as above, vanishing at $0$. There exists a unique piecewise $\mathscr{C}^1$ function $z$ with walues in $\RR^M$ solving, for some fixed $\bm{z}_0\in\RR^M$,
  \begin{equation*}
    \bm{z}(t) = \bm{z}_0 + \int_0^t \Delta_M\Big[\bm{z}(s)\odot \bm{\mu}(s)+\bm{f}(s)\Big]\,\dd s+\bm{x}(t).
  \end{equation*}
This function satisfies furthermore, for any $t\in [0,T]$,
\begin{multline}\label{ineq:lem:duasing}
 \|\bm{z}(t)\|_{-1,M}^2 + \int_{Q_t} \sigma_M(\bm{z} \odot \bm{\mu}^{1/2})(s,x)^2 \,\dd s\,\dd x  \\
     \leq  \|\bm{z}_0\|_{-1,M}^2 + \frac{1}{\alpha} \int_{Q_t} \sigma_M(\bm{f})(s,x)^2 \,\dd s\,\dd x +\int_{0}^t \meanM{\bm{z}(s)}^2 \meanM{\bm{\mu}(s)} \,\dd s\\  
    + \sum_{t_k\leq t}  \|\bm{a}_k\|_{-1,M}^2 
          + 2\sum_{t_k\leq t}  (\bm{z}(t_k^-) |\bm{a}_k)_{-1,M} + 2\int_{0}^t ( \bm{z}(s)|\bm{x}_R'(s))_{-1,M}\,\dd s.
  \end{multline}
\end{lem}
\begin{rmk}
Note that the two last terms of the r.h.s. in \eqref{ineq:lem:duasing} are of the same "nature" in the sense that replacing $\bm{x}_R$ by $\bm{x}_J$ in the integral, one formally recovers the corresponding discrete summation.
\end{rmk}
\begin{rmk}
Note also that  $\meanM{\bm{z}(t)}^2 \lesssim \|\bm{z}_0\|_{2,M}^2 + \|\bm{x}(t)\|_{2,M}^2$.
\end{rmk}

\begin{proof}
Given $\bm{z}_0$, $\bm{x}$ and $\bm{f}$, the uniqueness of such a function $\bm{z}$ is straightforward because taking the difference of two hypothetical solutions, one gets a linear homogeneous differential equation with continuous coefficients and $0$ as initial data. For the existence, we first note that if $\bm{x}_J=0$ the equations rewrites (after differentiation) as a simple linear ODE with continuous coefficients. Then, the equation being linear, we only need to treat the case when $\bm{z}_0=\bm{f}=\bm{x}_R =0$ and $m=1$ for which a solution is explicitely given by $\bm{z}(t) = \bm{z}_1\mathbf{1}_{t\geq t_1}$, where $\bm{z}_1$ is the (unique) solution in the case when $\bm{x}=\bm{f}=0$ and $\bm{z}_0=\bm{a}_1$. 
  Using Lemma~\ref{lem:duadis} we claim, for $0\leq k \leq m-1$, that the function $\bm{z}$ satisfies
\begin{multline*}
    \forall t\in I_k,\quad \|\bm{z}(t)\|_{-1,M}^2 + \int_{t_k}^t \int_{\T} \sigma_M(\bm{z} \odot \bm{\mu}^{1/2})(s,x)^2 \,\dd s\,\dd x  \\
    \leq  \|\bm{z}(t_k)\|_{-1,M}^2 + \int_{t_k}^t \meanM{\bm{z}(s)}^2 \meanM{\bm{\mu}(s)} \,\dd s\\ + \frac{1}{\alpha} \int_{t_k}^t \int_{\T} \sigma_M(\bm{f})(s,x)^2 \,\dd s\,\dd x 
    + 2\int_{t_k}^t ( \bm{z}(s)|\bm{x}_R'(s))_{-1,M}\,\dd s.
  \end{multline*}
  In particular, if $k\geq 1$, we have in particular at $t=t_{k}^-$
  \begin{multline*}
\|\bm{z}(t_{k}^-)\|_{-1,M}^2 + \int_{t_{k-1}}^{t_{k}} \int_{\T} \sigma_M(\bm{z} \odot \bm{\mu}^{1/2})(s,x)^2 \,\dd s\,\dd x  \\
    \leq  \|\bm{z}(t_{k-1})\|_{-1,M}^2 + \int_{t_{k-1}}^{t_{k}} \meanM{\bm{z}(s)}^2 \meanM{\bm{\mu}(s)} \,\dd s\\ + \frac{1}{\alpha} \int_{t_{k-1}}^{t_{k}} \int_{\T} \sigma_M(\bm{f})(s,x)^2 \,\dd s\,\dd x 
    + 2\int_{t_{k-1}}^{t_{k}} ( \bm{z}(s)|\bm{x}_R'(s))_{-1,M}\,\dd s.
  \end{multline*}
  Adding all these estimates down to $k=1$ we recover   
\begin{multline*}
    \forall t\in I_k,\quad \|\bm{z}(t)\|_{-1,M}^2 + \int_{0}^t \int_{\T} \sigma_M(\bm{z} \odot \bm{\mu}^{1/2})(s,x)^2 \,\dd s\,\dd x  \\
    \leq  \|\bm{z}_0\|_{-1,M}^2 + \sum_{j=1}^k \Big(\|\bm{z}(t_j)\|_{-1,M}^2 - \|\bm{z}(t_j^-)\|_{-1,M}^2\Big) +\int_{0}^{t} \meanM{\bm{z}(s)}^2 \meanM{\bm{\mu}(s)} \,\dd s\\ + \frac{1}{\alpha} \int_{0}^t \int_{\T} \sigma_M(\bm{f})(s,x)^2 \,\dd s\,\dd x 
    + 2\int_{0}^t ( \bm{z}(s)|\bm{x}_R'(s))_{-1,M}\,\dd s.
  \end{multline*}
  At this point it is important to note that the jumps of $\bm{z}$ are exactly the ones of $\bm{x}$, so that $\bm{z}(t_j) = \bm{a}_j + \bm{z}(t_j^-)$. Thus, 
  $\|\bm{z}(t_j)\|_{-1,M}^2 - \|\bm{z}(t_j^-)\|_{-1,M}^2=
  2(\bm{z}(t_j)|\bm{a}_j)_{-1,M} +
  \|a_j\|_{-1,M}^2$ 
and the proof is over. $\qedhere$
\end{proof}

\section{Quantitative estimates and proof of Theorem~\ref{thm:quantitativeresult}}\label{sec:quantification}
For a function $f$ defined on $[0,T]\times\T$, recalling the definition \eqref{eq:discT} of the discretized torus $\T_M$, we denote by $\widehat{\bm{f}}:[0,T]\rightarrow \RR^M$ the function whose value at time $t$ is the list of values of $f$ at the points $x_k\in \T_M$, for $1\leq k \leq M$. We have then the following proposition. 
\begin{prop}\label{prop:taylor}
  Let $u, v$ be two elements of $L^2(0,T;H^3(\T))$, solution of the system \eqref{eq:SKT:cons}. We have 
\begin{align*}
\partial_t  \widehat{\bm{u}}^M(t) &= 
{\Delta_M}\big[d_1 \widehat{\bm{u}}^{M}(t) + a_{12} \widehat{\bm{u}}^{M}(t)\odot
\widehat{\bm{v}}^{M}(t)\big] + \bm{r}^M(t),\\
\partial_t  \widehat{\bm{v}}^M(t) &= {\Delta_M}\big[d_1 \widehat{\bm{v}}^{M}(t) + a_{21} \widehat{\bm{v}}^{M}(t)\odot \widehat{\bm{u}}^{M}(t)\big] + \bm{s}^M(t),
\end{align*}
where the error terms $\bm{r}^M,\bm{s}^M:(0,T)\rightarrow\RR^M$ satisfy
\begin{align}
\label{bornu}
(\|\bm{r}^M\|_{\infty})_M \text{ and } (\|\bm{s}^M\|_{\infty})_M \operatorname*{\longrightarrow}_{M\rightarrow+\infty} 0,\text{ in }L^1(0,T).
\end{align}\end{prop}
\begin{proof}
For a smooth function $f$ defined on $\T$ we have, by Taylor expansion, for any $h\neq 0$ 
\begin{align*}
\frac{\tau_h f +\tau_{-h}f -2f}{h^2} = f'' + \textnormal{O}_{h\rightarrow 0}(h^2).
\end{align*}
A short computation shows that the operator $h^{-2}(\tau_h+\tau_{-h}-2\textnormal{Id})$ coincides with $\textnormal{D}_{-h}\textnormal{D}_h$ where $\textnormal{D}_h$ is the difference quotient operator defined Subsection~\ref{subsec:sobo} of the Appendix. In particular, using  Proposition~\ref{prop:sobo3}, we infer the following weakened equality if $f$ is only assumed to belong to $H^3(\T)$
\begin{align*}
\frac{\tau_h f +\tau_{-h}f -2f}{h^2} = f'' + \textnormal{o}_{h\rightarrow 0}(1),
\end{align*}
where the topology is still uniform but the rate of convergence is not \emph{a priori} controlled. Still thanks to Proposition~\ref{prop:sobo3}, in the previous equality $f$ can be replaced by the product of any pair of $H^3(\T)$ functions. In particular here, using this remark on $u$, $v$ and $uv$, we infer 
\begin{align*}
\partial_t u &= M^2(\tau_{1/M} + \tau_{-1/M}-2\textnormal{Id})\big[d_1 u + a_{12} uv\big] + r^M,\\
\partial_t v &= M^2(\tau_{1/M} + \tau_{-1/M}-2\textnormal{Id})\big[d_2 v + a_{21} uv\big] +s^M,
\end{align*}
where the error terms satisfy
\begin{align}
\label{bornu:cont}
(r_M)_M \text{ and }(s_M)_M \operatorname*{\longrightarrow}_{M\rightarrow+\infty} 0,\text{ in } L^1(0,T;L^\infty(\T)),
\end{align}
from which one deduces directly \eqref{bornu}. $\qedhere$
\end{proof}

On the other hand, we recall (see \eqref{eq:semimg_u}) that our stochastic process satisfies
\begin{align*}
  \bm{U}^{M,N}(t) &= \bm{U}^{M,N}(0) + \int_0^t \Delta_M\Bigl( d_1\bm{U}^{M,N}(s) + a_{12}\bm{U}^{M,N}(s)\odot \bm{V}^{M,N}(s) \Bigr) \,\dd s + \bm{\mathcal M}^{M,N}(t), \\
  \bm{V}^{M,N}(t) &= \bm{V}^{M,N}(0) + \int_0^t \Delta_M \Bigl( d_2\bm{V}^{M,N}(s) + a_{21}\bm{U}^{M,N}(s)\odot \bm{V}^{M,N}(s) \Bigr) \,\dd s + \bm{\mathcal N}^{M,N}(t),
\end{align*}
where $\bm{\mathcal M}^{M,N}$ is square integrable martingale whose quadratic variation is given by \eqref{eq:predquad} and $\bm{\mathcal N}^{M,N}$ satisfies similar properties.  
By symmetry, we can focus on the first species $\bm{U}^{M,N}$. For compactness of notation, we introduce a new  Poisson random measure $\mathcal N$  and its associated Poisson point process $\{(T_k,Y_k) : k\geq 1\}$ on $\RR_+\times E_M$. It consists in collecting the Poisson random measures $\mathcal N^j$ on the different sites and  the intensity of the new Poisson point process is $ds \otimes \nu_M(dy)$, where $E_M=\RR_+\times \{-1,1\}\times \{1,\ldots, M\}$, $\nu_M(d\rho, d\theta, di)=d\rho\otimes \beta(d\theta)\otimes n_M(di)$, and $n_M(di)=\sum_{1\leq j\leq M} \delta_j$ is the counting measure on the sites $\{1,\ldots, M\}$. The martingale $\bm{\mathcal M}^{M,N}$  can now be written as
\begin{align*}
\bm{\mathcal M}^{M,N}(t)=\sum_{k\geq 1} H(\bm{U}^{M,N}(T_k^-), \bm{V}^{M,N}(T_k^-),Y_k)1_{t\geq T_k}
-\int_0^t \phi(\bm{U}^{M,N}(s), \bm{V}^{M,N}(s))ds,
\end{align*}
where $H$ yields the jumps and $\phi$ the compensation
\begin{align}
    \label{notHphi}
H(\bm{u},\bm{v}, \rho,\theta,i)&= \1_{\rho \leq  2M^2Nu_i\bigl(d_1 + a_{12}v_i\bigr)} \frac{\bm{\mathrm{e}}_{i + \theta} - \bm{\mathrm{e}}_i}{N}, \quad    \phi(\bm{ u},\bm{v})=\Delta_M\Bigl( d_1 \bm{u} + a_{12} \bm{u}\odot \bm{v} \Bigr).
\end{align}
Denoting
\begin{align*}
  \bm{Z}^{M,N}(t) &= \widehat{\bm{u}}^M(t) - \bm{U}^{M,N}(t),
 \quad
  \bm{X}^{M,N}(t)= \int_0^t \bm{r}^M(s)\,\dd s - \bm{\mathcal M}^{M,N}(t),
\end{align*}
we have yet another system satisfied by these quantities
\begin{align}
  \bm{Z}^{M,N}(t) &= \bm{Z}^{M,N}(0) + \int_0^t {\Delta_M}\Bigl( \bm{Z}^{M,N}(s)\odot\bm{\Lambda}^{M,N}(s) + \bm{F}^{M,N}(s) \Bigr) \,\dd s + \bm{X}^{M,N}(t), \label{eq:evolz} 
\end{align}
where
\begin{align}
  \label{eq:defLambda}\bm{\Lambda}^{M,N}(t) &= d_1 \mathbf 1_M + a_{12}\bm{V}^{M,N}(t), \\
  \label{eq:defW}\bm{W}^{M,N}(t) &= \widehat{\bm{v}}^M(t) - \bm{V}^{M,N}(t),\\
  \label{eq:defF}\bm{F}^{M,N}(t) &= a_{12} \widehat{\bm{u}}^M\odot  \bm{W}^{M,N}(t).
\end{align}
Let us provide a useful estimate, which  allows to control the maritngale terms.
\begin{lem}\label{martinalagale} 
 For any $T>0$,
\begin{multline*}
    \EE\Bigl(\sup_{t\in[0,T]} \|\bm{\mathcal M}^{M,N}(t)\|_{-1,M}^2+\sup_{t\in[0,T]}\|\bm{\mathcal N}^{M,N}(t)\|_{-1,M}^2 \Bigr) \\\lesssim\frac{M^2}{N}\vertiii{\bm{Z}^{M,N}}_{T,M}^2 +\frac{M^2}{N}\vertiii{\bm{W}^{M,N}}_{T,M}^2+T \frac{M^2}{N}.
\end{multline*}
\end{lem}
\begin{proof}[Proof of Lemma~\ref{martinalagale}]
Without loss of generality we can focus on ${\bm{\mathcal{M}}}^{M,N}$. Doob inequality for square integrable martingales (see Corollary 6.2 in Chapter 1.6 in \cite{IW})  ensures 
\begin{align*}
      \EE\Bigl( \sup_{t\in[0,T]} \| \bm{\mathcal M}^{M,N}(t) \|_{2,M}^{2}\Bigr) \lesssim  \EE\left([\langle \bm{\mathcal M}^{M,N} \rangle(T)]_M\right).
      \end{align*}
      Owing to \eqref{borne-12}, it is sufficient to bound the r.h.s. of the previous inequality. For this purpose, we use \eqref{bornVQ}
and  get 
 \begin{align*}
  & \EE\left([\langle \bm{\mathcal M}^{M,N} \rangle(T)]_M\right) =\frac{1}{M}\sum_{i=1}^M\EE\Bigl(\langle \mathcal M_i^{M,N}\rangle^2(T) \Bigr)  \\
  & \qquad \qquad \lesssim \frac{1}{M} \frac{M^2}{N}  \int_0^T \EE\Bigl( \| \bm{U}^{M,N}(s)\|_1 + \|\bm{U}^{M,N}(s)\|_2^2 + \| \bm{V}^{M,N}(s)\|_2^2  \Bigr) \dd s .
  \end{align*}
Moreover, $\| \bm{U}^{M,N}(s)\|_1 =\| \bm{U}^{M,N}(0)\|_1$ a.s.
and we recall  that  $\bm{U}^{M,N}(t)=
  \widehat{\bm{u}}^M(t) - \bm{Z}^{M,N}(t)$ and  
  $\bm{V}^{M,N}(t)=
  \widehat{\bm{v}}^M(t) - \bm{W}^{M,N}(t)$ for any $s\geq 0$.
  Adding that  boundedness assumption on the solution of the SKT system and \eqref{condinitborn} (which guarantees \eqref{boundUV})
  ensure that
\[
  T\frac{M^2}{N}\| \bm{U}^{M,N}(0)\|_{1,M} + \frac{M^2}{N}\int_{Q_T}   \sigma_M\bigl(\widehat{\bm{u}}^{M}\bigr)^2+\sigma_M\bigl(\widehat{\bm{v}}^{M}\bigr)^2=T \mathcal{O}\bigl(\frac{M^2}{N}\bigr).
\]
Finally, we obtain
 \begin{align*}
 &\EE\left([\langle \bm{\mathcal M}^{M,N} \rangle(T)]_M\right) \\
 &\quad  \lesssim  \frac{M}{N}\int_0^T \EE\bigl( \|\bm{Z}^{M,N}(s)\|_2^2 + \| \bm{W}^{M,N}(s)\|_2^2 \bigr) \, \dd s +T\frac{M^2}{N} \\
  & \quad \lesssim \frac{M^2}{N} \int_{Q_T}  \EE\left( \sigma_M\bigl(\bm{Z}^{M,N}\bigr)(s,x)^2+\sigma_M\bigl(\bm{W}^{M,N}\bigr)(s,x)^2\right)  \,\dd s\,\dd x+T \frac{M^2}{N},
\end{align*}
where we recall that $\|\bm{ u}\|_{2}^2=M\|\bm{ u}\|_{2,M}^2=M\| \sigma_M(\bm{ u}\|_{L^2([0,T])}^2 $. It
 ends the proof recalling definition \eqref{normecheloudiscrete}.
\end{proof}

We can now apply the discrete duality lemma obtained in the previous section to
control the gap $\bm{Z}^{M,N}$. This is the core of the next result and  yields Theorem \ref{thm:quantitativeresult}.
Given a process $\bm{Z} \colon \Omega \times [0,T]\to \mathbb R^M$ defined on a probability space 
$(\Omega, \mathcal F, \mathbb P)$, we consider the discrete analog of the norm $\vertiii{\cdot}_T$ introduced in \eqref{eq:tripnorm}, that is:
\begin{align}
\label{normecheloudiscrete}
  \vertiii{\bm{Z}}_{T,M} \coloneqq \left(\sup_{t\in[0,T]}\EE\left(\|\bm{Z}(t)\|_{-1,M}^2\right) + \EE\left(\Vert\sigma_M (\bm{Z}) \Vert_{L^2(Q_T)}^2 \right)\right)^{1/2}.
\end{align}
\begin{prop}\label{prop:core}
  Let $u, {v}$ be a bounded $L^2(0,T;H^3(\T))$ non-negative solution of the system \eqref{eq:SKT:cons} satisfying \eqref{ineq:small}. There exists constant $\textnormal{C},\textnormal{D}>0$ depending only on the diffusion parameters and $\|u\|_{L^\infty(Q_T)}\|v\|_{L^\infty(Q_T)}$ such that for any $(M,N)\in \mathbb N^2$ satisfying $N\geq M^2\textnormal{D}$, there holds
      \begin{align}\label{eq:coreineq}
           \vertiii{\bm{Z}^{M,N}}_{T,M}^2 + \vertiii{\bm{W}^{M,N}}_{T,M}^2  
         \leq \textnormal{C}\Big(  \EE(A_{T,M,N}(0)) 
          + T \frac{M^2}{N} + \delta_M\Big),
        \end{align}
    where $\vertiii{\cdot}_{T,M}$ is defined by \eqref{normecheloudiscrete}, $(\delta_M)_M\rightarrow 0$ and 
    \begin{multline}\label{eq:defATMN}
        A_{T,M,N}(0):=\|\bm{Z}^{M,N}(0)\|_{-1,M}^2 + T\meanM{\bm{Z}^{M,N}(0)}^2 \meanM{\bm{\Lambda}^{M,N}(0)} 
        \\ + \|\bm{W}^{M,N}(0)\|_{-1,M}^2 + T\meanM{\bm{W}^{M,N}(0)}^2 \meanM{\bm{\Gamma}^{M,N}(0)}.
    \end{multline}
\end{prop}
\begin{rmk}
\label{rem:delta}
The sequence $\delta_M$ is directly linked to the error terms $\bm{r}^M$ and $\bm{s}^M$ introduced in  Proposition~\ref{prop:taylor}. From the proof of this very proposition, it is therefore clear that assuming more regularity for $u,v$, one can give an explicit rate of convergence for $\delta_M$.  For instance if $u,v$ is assumed $L^2(0,T;\mathscr{C}^4(\T))$ one could take $\delta_M=\textnormal{O}(1/M^2)$. 
\end{rmk}

\begin{proof}[Proof of Proposition \ref{prop:core}]
 We apply Lemma \ref{lem:duasing} with $\bm{z}\coloneqq \bm{Z}^{M,N}$ and
    $\bm{x}=\bm{x}_R+\bm{x}_J \coloneqq \bm{X}^{M,N}$ and
     $\bm{f}\coloneqq\bm{F}^{M,N}$ and $\bm{\mu}\coloneqq \bm{\Lambda}^{M,N}$
     recalling the definitions \eqref{eq:defLambda} -- \eqref{eq:defF}.
More explicitely, recalling \eqref{notHphi}, we have here
    \begin{align*}
        {\bm x}_R(t)&=\int_0^t r^M(s)ds+\int_0^t \phi(\bm{U}^{M,N}(s), \bm{V}^{M,N}(s))ds, \\ {\bm x}_J(t)&
        =-\sum_{k\geq 1} H(\bm{U}^{M,N}(T_k^-), \bm{V}^{M,N}(T_k^-),Y_k)1_{t\geq T_k},
        \end{align*}
        where  we recall that $\{(T_k,Y_k) : k\geq 0\}$ is a Poisson point process  on $\RR_+\times E_M$ with intensity $ds \otimes \nu_M(dy)$.
Besides, using that $t \mapsto \meanM{\bm{\Lambda}^{M,N}(t)}$ and $t \mapsto \meanM{\bm{Z}^{M,N}(t)}$ are constant functions, we observe
$$\int_{0}^t \meanM{\bm{z}(s)}^2 \meanM{\bm{\mu}(s)} \,\dd s=T\meanM{\bm{Z}^{M,N}(0)}^2  \meanM{\bm{\Lambda}^{M,N}(0)}.$$
We obtain from Lemma \ref{lem:duasing} that for any $t\leq T$ that
    \begin{align}\label{eq:normZ}
        & \|\bm{Z}^{M,N}(t)\|_{-1,M}^2 + \int_{Q_t} \sigma_M\bigl(\bm{Z}^{M,N} \odot (\bm{\Lambda}^{M,N})^{1/2}\bigr)(s,x)^2 \,\dd s \dd x \notag\\
        & \qquad  \quad  \leq \|\bm{Z}^{M,N}(0)\|_{-1,M}^2 + \frac{1}{d_1} \int_{Q_t} \sigma_M(\bm{F}^{M,N})(s,x)^2 \,\dd s\dd x \notag\\
        & \qquad\qquad \qquad  + T\meanM{\bm{Z}^{M,N}(0)}^2  \meanM{\bm{\Lambda}^{M,N}(0)}+2 \int_{0}^t ( \bm{Z}^{M,N}(s) \, | \, \bm{r}^M(s) )_{-1,M}\,\dd s\nonumber \\  
        & \qquad\qquad \qquad \qquad +{\mathcal R}^{M,N}(t),
    \end{align}
    where  $\mathcal{R}^{M,N}(t)$ is given by
    \begin{multline}\label{def:RMN}
        \mathcal{R}^{M,N}(t)    
        = \sum_{T_k\leq t} \|H(\bm{U}^{M,N}(T_k^-), \bm{V}^{M,N}(T_k^-),Y_k)\|_{-1,M}^2\\
         -  2\sum_{T_k \leq t}  (\bm{Z}^{M,N}(T_k^-) \, | \, H(\bm{U}^{M,N}(T_k^-), \bm{V}^{M,N}(T_k^-),y))_{-1,M}\\
         \qquad + 2\int_{0}^t ( \bm{Z}^{M,N}(s) \, | \, \phi(\bm{U}^{M,N}(s), \bm{V}^{M,N}(s)))_{-1,M}\,\dd s.
    \end{multline}
    Some cancellations will happen for the error term $\mathcal{R}^{M,N}(t)$ when taking the expectation, thanks to the martingale structure. For the moment we keep it as it and focus on the other terms.  Besides, recalling \eqref{borne-12}, we have 
 $\| u\|_{-1,M}\leq \| u\|_{2,M}\leq \| u\|_{\infty}$  and  Cauchy Schwarz inequality entails
 for any $s\geq 0$,
   \begin{align*}
 \vert ( \bm{Z}^{M,N}(s) \, | \, \bm{r}^M(s) )_{-1,M}\vert &\leq
 \| \bm{Z}^{M,N}(s)\|_{-1,M} \| \bm{r}^M(s) \|_{\infty}.
   \end{align*}
We plug this estimate in \eqref{eq:normZ}.
We also use  that $\Lambda_i^{M,N} \geq d_1$ and that
    \[
        |\sigma_M(\bm{F}^{M,N})(s,x)| \leq a_{12} \| {u} \|_{L^\infty(Q_T)}|\sigma_M(\bm{W}^{M,N})(s,x)|,
    \]
    as $\hat{\bm{u}}^M$ takes the values of $u$ in the grid. We  obtain
    \begin{multline*}
    \|\bm{Z}^{M,N}(t)\|_{-1,M}^2 +d_1 \int_{Q_t} \sigma_M(\bm{Z}^{M,N})(s,x)^2 \,\dd s \dd x 
        \\
        \leq   \|\bm{Z}^{M,N}(0)\|_{-1,M}^2 + T\meanM{\bm{Z}^{M,N}(0)}^2 \meanM{\bm{\Lambda}^{M,N}(0)} \\
         \qquad + \frac{(a_{12}\|{u}\|_{L^\infty(Q_T)})^2}{d_1}\int_{Q_t} \sigma_M(\bm{W}^{M,N})(s,x)^2 \,\dd s\dd x \\
         + 2\int_0^t \|\bm{Z}^{M,N}(s)\|_{-1,M}\|\bm{r}^M(s)\|_\infty\,\dd s + \mathcal{R}^{M,N}(t).
    \end{multline*}
    As the roles of $\bm{Z}^{M,N}$ and $\bm{W}^{M,N}$ are symmetric in the previous inequality, we have a similar estimate for $\bm{W}^{M,N}$. Thus, by setting
    \[
        \bm{\Gamma}^{M,N}(t) = d_2 + a_{21}\bm{U}^{M,N}(t),
    \]
    and defining $\mathcal{S}^{M,N}(t)$ as $\mathcal{R}^{M,N}(t)$ (exchanging $\bm{U}^{M,N}$ and $\bm{V}^{M,N}$ and replacing $\bm{Z}^{M,N}$ by $\bm{W}^{M,N}$ in \eqref{def:RMN}) we get
    \begin{multline*}
    \|\bm{W}^{M,N}(t)\|_{-1,M}^2 +d_2 \int_{Q_t} \sigma_M(\bm{W}^{M,N})(s,x)^2 \,\dd s \dd x 
        \\
        \leq   \|\bm{W}^{M,N}(0)\|_{-1,M}^2 + T\meanM{\bm{W}^{M,N}(0)}^2 \meanM{\bm{\Gamma}^{M,N}(0)} \\
         + \frac{(a_{21}\|{v}\|_{L^\infty(Q_T)})^2}{d_2}\int_{Q_t} \sigma_M(\bm{Z}^{M,N})(s,x)^2 \,\dd s\dd x \\
         + 2\int_0^t \|\bm{W}^{M,N}(s)\|_{-1,M}\|\bm{s}^M(s)\|_\infty\,\dd s + \mathcal{S}^{M,N}(t).
    \end{multline*}
    Plugging now this inequality in the estimate for $\bm{Z}^{M,N}$ gives us
    \begin{align*}
        &\|\bm{Z}^{M,N}(t)\|_{-1,M}^2 + d_1\int_{Q_t} \sigma_M(\bm{Z}^{M,N})(s,x)^2 \,\dd s \dd x \\
        &\qquad \leq \left( 1+\frac{(a_{12}\|{u}\|_{L^\infty(Q_T)})^2}{d_1d_2}\right) \, A_{T,M,N}(0) \\ 
        &\qquad\qquad + \frac{1}{d_1}\Bigl(\frac{a_{12}a_{21}\|{u}\|_{L^\infty(Q_T)}\|{v}\|_{L^\infty(Q_T)}}{d_2}\Bigr)^2  \int_{Q_t} \sigma_M(\bm{Z}^{M,N})(s,x)^2 \,\dd s\dd x  \\
        &\qquad \qquad + \frac{(a_{12}\|{u}\|_{L^\infty(Q_T)})^2}{d_1d_2}  \biggl( 2\int_0^t \|\bm{W}^{M,N}(s)\|_{-1,M}\|\bm{s}^M(s)\|_\infty\,\dd s + \mathcal{S}^{M,N}(t)\biggr) \\
        &\qquad \qquad +  2\int_0^t \|\bm{Z}^{M,N}(s)\|_{-1,M}\|\bm{r}^M(s)\|_\infty\,\dd s + \mathcal{R}^{M,N}(t).
    \end{align*}
    By using our bound \eqref{ineq:small} on $\|u\|_{L^\infty(Q_T)} \|v\|_{L^\infty(Q_T)}$,
    we can absorb the term of the third line
    in the l.h.s. of the inequality. Thus, letting $\lesssim$ to depend on these (deterministic and fixed) parameters 
    this yields
    \begin{multline}\label{ineq:rhs}
        \|\bm{Z}^{M,N}(t)\|_{-1,M}^2 + d_1\int_{Q_t} \sigma_M(\bm{Z}^{M,N})(s,x)^2 \,\dd s \dd x \\
         \lesssim  A_{T,M,N}(0)  
         + \mathcal{R}^{M,N}(t) + \mathcal{S}^{M,N}(t) \\
        + \int_0^t \|\bm{Z}^{M,N}(s)\|_{-1,M}\|\bm{r}^M(s)\|_\infty\,\dd s
        \\+\int_0^t \|\bm{W}^{M,N}(s)\|_{-1,M}\|\bm{s}^M(s)\|_\infty\,\dd s,
    \end{multline}
    where $A_{T,M,N,}(0)$ is defined in \eqref{eq:defATMN}.     As before, the r.h.s. of \eqref{ineq:rhs} is invariant exchanging $\bm{Z}^{M,N}$ and $\bm{W}^{M,N}$ so that the same estimate holds for $\bm{W}^{M,N}$ in the l.h.s. (with $d_2$ instead of $d_1$). We will now sum both inequalities and take expectation. For the sake of clarity, we therefore introduce
    \begin{align*}
        \psi(t) &:= \EE\Big(\|\bm{Z}^{M,N}(t)\|_{-1,M}^2\Big) +  \EE\Big(\|\bm{W}^{M,N}(t)\|_{-1,M}^2\Big),\\
        \theta(t) &:= \EE\Big(\int_{Q_t} \sigma_M(\bm{Z}^{M,N})^2\Big) + \EE\Big(\int_{Q_t} \sigma_M(\bm{W}^{M,N})^2\Big),
    \end{align*}
    and we thus infer 
    \begin{multline}\label{ineq:psitheta}
        \psi(t) + \theta(t) \lesssim \EE(A_{T,M,N}(0)) + \EE\big(\mathcal{R}^{M,N}(t)+\mathcal{S}^{M,N}(t)\big) \\
        +\int_0^t \EE\Big(\|\bm{Z}^{M,N}(s)\|_{-1,M}\Big)\|\bm{r}^M(s)\|_\infty\,\dd s \\+\int_0^t \EE\Big(\|\bm{W}^{M,N}(s)\|_{-1,M}\Big)\|\bm{s}^M(s)\|_\infty\,\dd s.
    \end{multline}
    Now, on one hand, thanks to Cauchy-Schwarz inequality we have 
     \begin{align*}
         \EE\Big(\|\bm{Z}^{M,N}(s)\|_{-1,M}\Big) \leq \left(\EE\big(\|\bm{Z}^{M,N}(s)\|_{-1,M}^2\big)\right)^{1/2},
     \end{align*}
     and a similar majoration replacing $\bm{Z}^{M,N}$ by $\bm{W}^{M,N}$, so that
     \begin{align*}
         \EE\Big(\|\bm{Z}^{M,N}(s)\|_{-1,M}\Big) + \EE\Big(\|\bm{W}^{M,N}(s)\|_{-1,M}\Big) \lesssim \psi(s)^{1/2}.
     \end{align*}
     On the other hand, going back to the definition \eqref{def:RMN} of $\mathcal{R}^{M,N}$, one checks that the two last terms on the right hand side form a martingale starting from $0$. Then   taking expectation in \eqref{def:RMN}, the two last term disappear and  we recover 
    \begin{align*}
        \EE\left( {\mathcal R}^{M,N}(t)\right) 
        &=\EE\left(\sum_{T_k\leq t} \|H(\bm{U}^{M,N}(T_k^-), \bm{V}^{M,N}(T_k^-),Y_k)\|_{-1,M}^2 ) \right) =\EE\left(\|\bm{\mathcal M}^{M,N}(t)\|_{-1,M}^2  \right),
    \end{align*}
    where the last identity can be directly obtained from the semimartingale decomposition of $\|\bm{\mathcal M}^{M,N}(t)\|_{-1,M}^2$ (as for the classical proof with canonical euclidian inner product). The same bound applying for $\mathcal{S}^{M,N}(t)$ only replacing the martingale $\bm{\mathcal{M}}^{M,N}$ by $\bm{\mathcal{N}}^{M,N}$. All in all, we infer from \eqref{ineq:psitheta} the following estimate for any $t\leq T$
         \begin{align*}
        \psi(t) + \theta(t) \lesssim C_0 
        +\int_0^t \psi(s)^{1/2} \Big(\|\bm{r}^M(s)\|_\infty + \|\bm{s}^M(s)\|_\infty\Big)\,\dd s,
    \end{align*}
    where 
    \begin{align}\label{eq:defC0}
        C_0 = \EE(A_{T,M,N}(0)) + \sup_{t\in[0,T]} \EE\left(\|\bm{\mathcal M}^{M,N}(t)\|_{-1,M}^2 \right) + \sup_{t\in[0,T]} \EE\left(\|\bm{\mathcal N}^{M,N}(t)\|_{-1,M}^2 \right). 
    \end{align}
          Using the nonlinear Gronwall Lemma~\ref{lem:gron} we infer 
    \begin{align}
        \sup_{t\in[0,T]}\Big(\psi(t) + \theta(t)\Big) \lesssim C_0 +\Big(\int_0^T (\|\bm{r}^M(s)\|_\infty+\|\bm{s}^M(s)\|_\infty)\,\dd s\Big)^2.
    \end{align}
    Since $(\|\bm{r}^M\|_\infty)_M$ and $(\|\bm{s}^M\|_\infty)_M$ both converge to $0$ in $L^1(0,T)$ (see \eqref{bornu}), the squared integral in the third line is a sequence $(\delta_M)_M\rightarrow 0$. We are left  with controlling the  two martingale terms that appear in the definition \eqref{eq:defC0} of $C_0$.
    This is achieved by Lemma~\ref{martinalagale} and  
\eqref{ineq:psitheta} finally yields
     \begin{multline*}
           \vertiii{\bm{Z}^{M,N}}_{T,M}^2 + \vertiii{\bm{W}^{M,N}}_{T,M}^2  
         \lesssim  \EE(A_{T,M,N}(0)) 
          + T \frac{M^2}{N} + \delta_M \\+ \frac{M^2}{N}\Big[\vertiii{\bm{Z}^{M,N}}_{T,M}^2 + \vertiii{\bm{W}^{M,N}}_{T,M}^2\Big].
        \end{multline*}    
The previous inequality can be rewritten replacing $\lesssim$ by $\leq \frac{\textnormal{D}}{2}$ for some $\textnormal{D}>2$. If indeed $N\geq M^2 D$, then $1-\frac{\textnormal{D}M^2}{2N}\geq \frac12$ and the terms in the last line can be absorbed in l.h.s.
\end{proof}

Now we can prove the remaining main result.

\begin{proof}[Proof of Theorem \ref{thm:quantitativeresult}]
We have
\begin{align*}
  \zeta^{M,N} &\coloneqq \pi_M(\bm{U}^{M,N})-u\\
  &\,=\pi_M(\bm{U}^{M,N}-\widehat{\bm{u}}^M) + \pi_M(\widehat{\bm{u}}^M)-u=\pi_M(\bm{Z}^{M,N})+\iota_M(u)-u,
\end{align*}
where the interpolation operator $\iota_M$ is the one used in Lemma~\ref{lem:bh}. Using the triangular inequality, we infer
  \begin{multline}
    \sup_{t\in[0,T]}\EE\Big(\|\zeta^{M,N}(t)\|_{H^{-1}(\T)}^2\Big)+\EE\left(\|\zeta^{M,N}\|_{L^2(Q_T)}^2\right)\\  \leq  \sup_{t\in[0,T]}\EE\Big(\|\pi_M(\bm{Z}^{M,N})(t)\|_{H^{-1}(\T)}^2\Big)+\EE\left(\|\pi_M(\bm{Z}^{M,N})\|_{L^2(Q_T)}^2\right)\\\label{ineq:believeproof}+ \sup_{t\in[0,T]}\|\iota_M(u)-u\|_{H^{-1}(\T)}^2+\|\iota_M(u)-u\|_{L^2(Q_T)}^2.
\end{multline}
Now, using Proposition~\ref{prop:normlp} we have that $\|\pi_M(\bm{Z}^{M,N})\|_{L^2(Q_T)} \leq \|\sigma_M(\bm{Z}^{M,N})\|_{L^2(Q_T)}$, and using the equivalence \eqref{sim:equiv} of Proposition~\ref{prop:defsobneg}, we get for all $t\in[0,T]$
\begin{align*}
  \|\pi_M(\bm{Z}^{M,N})(t)\|_{H^{-1}(\T)}&\lesssim \|\bm{Z}^{M,N}(t)\|_{-1,M} + M^{-1} \|\pi_M(\bm{Z}^{M,N}(t))\|_{L^2(\T)}\\
  & \lesssim\|\bm{Z}^{M,N}(t)\|_{-1,M} + M^{-1}\|\sigma_M(\bm{Z}^{M,N}(t)\|_{L^2(\T)}. 
\end{align*}
Then the  expectation terms in the r.h.s. of \eqref{ineq:believeproof} satisfy the following bound for $M\geq1$
\begin{equation*}
  \sup_{t\in[0,T]}\EE\left(\|\pi_M(\bm{Z}^{M,N})(t)\|_{H^{-1}(\T)}^2\right)+\EE\left(\|\pi_M(\bm{Z}^{M,N})\|_{L^2(Q_T)}^2\right) \lesssim_T \ \vertiii{\bm{Z}^{M,N}}_{T,M}^2, 
  \end{equation*}
  where $\vertiii{\cdot}_{T,M}$ is defined in \eqref{normecheloudiscrete}.
  Using Proposition~\ref{prop:core}, we infer
          \begin{align*}
  \sup_{t\in[0,T]}\EE\left(\|\pi_M(\bm{Z}^{M,N})(t)\|_{H^{-1}(\T)}^2\right)+\EE\left(\|\pi_M(\bm{Z}^{M,N})\|_{L^2(Q_T)}^2\right) \lesssim   \EE(A_{T,M,N}(0)) 
          + T \frac{M^2}{N} + \delta_M,
  \end{align*}
  where $(\delta_{M})_M\rightarrow 0$. Recalling the definition \eqref{eq:defATMN} of $A_{T,M,N}(0)$, we use $\meanM{\bm{Z}^{M,N}(0)}^2 \leq \|\bm{Z}^{M,N}(0)\|_{-1,M}^2$ and that $\meanM{\Lambda^{M,N}(0)} = d_1 + a_{12}\meanM{\bm{V}^{M,N}(0)}$ is uniformly bounded almost surely thanks to \eqref{condinitborn} (with similar estimates for the second species) to infer  
  \begin{multline*}
  \sup_{t\in[0,T]}\EE\left(\|\pi_M(\bm{Z}^{M,N})(t)\|_{H^{-1}(\T)}^2\right)+\EE\left(\|\pi_M(\bm{Z}^{M,N})\|_{L^2(Q_T)}^2\right) \\ \lesssim   \EE\left(\|\pi_M(\bm{Z}^{M,N})(0)\|_{H^{-1}(\T)}^2\right) + \EE\left(\|\pi_M(\bm{W}^{M,N})(0)\|_{H^{-1}(\T)}^2\right) 
          + \frac{M^2}{N} + \delta_M.
  \end{multline*}
   Of course we can replace $\bm{Z}^{M,N}(0)$ by $\zeta^{M,N}(0)$ but with an extra cost of $\|\iota_M(u_0)-u_0\|_{H^{-1}(\T)}^2$, that we somehow already had looking at the r.h.s. of \eqref{ineq:believeproof}. For this, we invoke Lemma~\ref{lem:bh} which allow us to write
  \begin{align*}
    \sup_{t\in[0,T]}\|\iota_M(u)-u\|_{H^{-1}(\T)}^2+\|\iota_M(u)-u\|_{L^2(Q_T)}^2 \lesssim M^{-4} \|u\|_{L^\infty\cap L^2([0,T];H^2(\T))}^2,
    \end{align*}
which can be added to the sequence $(\delta_M)_M$ going to $0$. Proceeding similarly to get the control on $\bm{W}^{M,N}$ and gathering all the terms leads to the conclusion.
\end{proof}

\section{Appendix}

\subsection{Nonlinear Gronwall lemma}
\begin{lem}\label{lem:gron}
Assume $0\leq \psi,\theta \in \mathscr{C}^0(\mathbb{R}_+)$ satisfy for some constant $C_0\geq 0$ and $0\leq c\in L_{\textnormal{loc}}^1(\mathbb{R}_+)$
\begin{align*}
    \psi(t) + \theta(t) \leq C_0 + 2 \int_0^t c(s)\psi(s)^{1/2}\,\dd s.
\end{align*}
Then there holds
\begin{align*}
    \psi(t) + \theta(t) \leq \Big(C_0^{1/2} +  \int_0^t c(s)\,\dd s\Big)^2.
\end{align*}
\end{lem}
\begin{proof}
W.l.o.g we can assume $C_0>0$ (or replace it by $C_0+\varepsilon$ and let $\varepsilon\rightarrow 0$). In that case the function \[a(t):=C_0+2\int_0^t c(s)\psi(s)^{1/2}\,\dd s\]
satisfies $a'(t) = 2 c(t) \psi(t)^{1/2}\leq 2 c(t) a(t)^{1/2}$ which integrates  as \[a(t)^{1/2}-C_0^{1/2}\leq \int_0^t c(s)\,\dd s,\]
since $a(t)\geq a(0)=C_0>0$. The conclusion follows using $\psi(t)+\theta(t)\leq a(t)$. $\qedhere$
\end{proof}
\subsection{Difference quotient and Sobolev spaces}\label{subsec:sobo}
For a real number $h\neq 0$, we consider the difference quotient operator \begin{align*}
      \textnormal{D}_h : L^1(\T)&\longrightarrow L^1(\T), \\
       f&\longmapsto \frac{\tau_h f - f}{h}.
       \end{align*}
We recall the following standard result for Sobolev spaces 
\begin{prop}\label{prop:sobo1}
  For $f\in H^1(\T)$, one has for any $h\neq 0$, $\|\textnormal{D}_h f\|_2 \leq \|f'\|_2$. In particular, $(\textnormal{D}_h f)_h\rightarrow f'$ in $L^2(\T)$, as $h\rightarrow 0$.
\end{prop}
The estimate is (for instance) proven in \cite[Proposition 9.3]{brezis} and the convergence is obtained by a straightforward density argument. From the previous result we recover the following one. 
\begin{prop}\label{prop:sobo2}
  For $f\in H^2(\T)$ has $(\textnormal{D}_{-h}\textnormal{D}_h f)_h \rightarrow f''$ in $L^2(\T)$ as $h\rightarrow 0$.
\end{prop}
\begin{proof}
We use Proposition~\ref{prop:sobo1} for both $f$ and $f'$, writing 
\begin{align*}
    \textnormal{D}_{-h}\textnormal{D}_h f - f'' = \textnormal{D}_{-h} (\textnormal{D}_h f -f') + (\textnormal{D}_{-h} f' - f''),
\end{align*}
and the conclusion follows since the family of operators $(\textnormal{D}_h)_{h\neq 0}$ is uniformly bounded from $H^1(\T)$ to $L^2(\T)$. $\qedhere$ 
\end{proof}
With the previous, one gets eventually the following result. 
\begin{prop}\label{prop:sobo3}
  For $f,g \in H^3(\T)$, on has $(\textnormal{D}_{-h}\textnormal{D}_h f)_h \rightarrow f''$ in $L^\infty(\T)$ as well as $(\textnormal{D}_{-h}\textnormal{D}_h fg)_h \rightarrow (fg)''$ in $L^\infty(\T)$, as $h\rightarrow 0$.
\end{prop}
\begin{proof}
For the first convergence, one just note that differentiation commutes with difference quotients so that from Proposition~\ref{prop:sobo2} one directly infers $(\textnormal{D}_{-h}\textnormal{D}_h f)_h \rightarrow f''$ in $H^1(\T)\hookrightarrow \mathscr{C}^0(\T)$ and the uniform limit follows. For the product, we simply use that $H^3(\T)$ is an algebra (see \emph{e.g.} \cite[Theorem 4.39]{adams} for a proof). $\qedhere$ 
\end{proof}

\subsection{Discrete--continuous dictionnary}

\begin{center}
\begin{tabular}{c|c}
  Discrete & Continuous\\
  \hline
  $\Delta_M$ & $\Delta$ \\
  $\|\cdot\|_{p,M}$ & $\|\cdot\|_{L^p(\T)}$\\
  $(\cdot|\cdot)_M$ & $(\cdot|\cdot)_{L^2(\T)}$\\
  $\|\cdot\|_{-1,M}$ & $\|\cdot\|_{H^{-1}(\T)}$\\
  $\vertiii{\cdot}_{T,M}$ &   $\vertiii{\cdot}_{T}$\\
  $\meanM{\cdot}$ & $\mean{\cdot}$
\end{tabular}
\end{center}

\bigskip

\paragraph*{Acknowledgements}
The authors thank the two anonymous referees for their valuable remarks which led to a substantial improvement of the article. They also thank Joaquin Fontbona for stimulating discussions on this topic. A.~M. would like to thank Ariane Trescases for a fruitful discussion which took place at the Tsinghua University in China, back in April 2015, during which the core idea of Theorem~\ref{thm:sktstab} was discovered, even though the proof was not yet formalized (\cite{moussa_nonloc_tri} did not exist back then!). F. M.-H. acknowledges financial support received under the Doctoral Fellowship ANID-PFCHA/Doctorado Nacional/2017-21171912 and thanks support from Millennium Nucleus Stochastic Models of Complex and Disordered Systems from Millennium Scientific Initiative. This work  was partially funded by the Chair "Mod\'elisation Math\'ematique et Biodiversit\'e" of VEOLIA-Ecole Polytechnique-MNHN-F.X and ANR ABIM 16-CE40-0001.


\end{document}